\documentclass[10pt]{amsart}

\usepackage{graphicx} 
\usepackage{amsmath}
\usepackage{mathtools}

\usepackage[initials]{amsrefs}

\usepackage{pstricks}
\usepackage{enumerate}

\usepackage{graphicx}
\usepackage{verbatim}
\usepackage{bm}

\usepackage{hyperref}
\usepackage{mathrsfs}
\usepackage{amssymb}
\usepackage{latexsym}

\usepackage[mathscr]{eucal}
\usepackage{slashed}

\usepackage[all]{xy}
\usepackage{pb-diagram,pb-xy}
\usepackage[a4paper, hmargin=3.0cm, tmargin=3.2cm, bmargin=3.2cm]{geometry}

\numberwithin{equation}{section}
 
\theoremstyle{definition}
\newtheorem{defn}[equation]{Definition}

\newtheorem{definition-proposition}[equation]{Definition/Proposition}
\newtheorem{assumption}[equation]{Assumptions}

\newtheorem{convention}[equation]{Convention}
\theoremstyle{remark}
\newtheorem{rem}[equation]{Remark}
\newtheorem{remark}[equation]{Remark}

\newtheorem{example}[equation]{Example}

\theoremstyle{plain}
\newtheorem{thm}[equation]{Theorem}
\newtheorem{lem}[equation]{Lemma}
\newtheorem{proposition}[equation]{Proposition}

\newtheorem{cor}[equation]{Corollary}
\newcommand{\lra}{\longrightarrow}

\newcommand{\bN}{\mathbb{N}}

\newcommand{\bR}{\mathbb{R}}

\newcommand{\bZ}{\mathbb{Z}}

\newcommand{\cC}{\mathcal{C}}
\newcommand{\cD}{\mathcal{D}}
\newcommand{\cB}{\mathcal{B}}

\newcommand{\cA}{\mathcal{A}}

\newcommand{\cL}{\mathcal{L}}

\newcommand{\Aut}{\operatorname{Aut}}

\DeclareMathOperator*{\colim}{colim}

\newcommand{\Sing}{\mathrm{Sing}}

\newcommand{\eps}{\epsilon}

\newcommand{\inter}[1]{\mathrm{int}(#1)} 

\newcommand{\id}{\mathrm{Id}}
\newcommand{\inc}{\operatorname{inc}}

\newcommand{\lift}{\mathrm{Lift}}

\newcommand{\inj}{\mathrm{inj}}
\newcommand{\Top}{\mathrm{Top}}
\newcommand{\Set}{\mathrm{Set}}

\newcommand{\Fun}{\mathrm{Fun}}

\newcommand{\Mor}{\operatorname{Mor}}

\newcommand{\Ob}{\operatorname{Ob}}
\newcommand{\hofib}{\operatorname{hofib}}
\newcommand{\norm}[1]{\| #1 \|}
\newcommand{\abs}[1]{| #1 |}

\newcommand{\si}{\mathrm{s}}
\newcommand{\se}{\mathrm{ss}}
\newcommand{\bis}{\mathrm{bis}}

\title[Semi-simplicial spaces]{Semi-simplicial spaces}

\author{Johannes Ebert}

\address{Mathematisches Institut, Universit\"at M\"unster\\
Einsteinstra{\ss}e 62\\
48149 M\"unster\\
Germany}

\email{johannes.ebert@uni-muenster.de}

\author{Oscar Randal-Williams}

\address{DPMMS\\
Centre for Mathematical Sciences\\
Wilberforce Road\\
Cambridge CB3 0WB\\
UK}

\email{or257@cam.ac.uk}

\date{\today}

\begin{document}

\begin{abstract}
This is an exposition of homotopical results on the geometric realization of semi-simplicial spaces. We then use these to derive basic foundational results about classifying spaces of topological categories, possibly without units. The topics considered include: fibrancy conditions on topological categories; the effect on classifying spaces of freely adjoining units; approximate notions of units; Quillen's Theorems A and B for non-unital topological categories; the effect on classifying spaces of changing the topology on the space of objects; the Group-Completion Theorem. 
\end{abstract}

\maketitle

\setcounter{tocdepth}{1}
\tableofcontents

Semi-simplicial spaces play an important role in the theory of moduli spaces of manifolds, beginning with \cite{GRW1} and \cite{RWStab}, and continuing in \cite{GRW14}, \cite{GRWHomStab2}, \cite{GRWHomStab1}. In those papers, a number of key properties of semi-simplicial spaces are used. While such results are mostly well-known to experts, a consistent exposition seems to be missing. The first goal of the present note is to give such an exposition; we hope that it helps to make the basic technology of those papers more understandable to the non-expert. Results which are repeatedly used in \cite{GRW1}, \cite{GRW14}, \cite{GRWHomStab2} and \cite{GRWHomStab1} are stated in this paper as Theorem \ref{thm:levelwiseequivalence}, Theorem \ref{lem:simplicial-hocartesianness} and Lemma \ref{lem:connecttivity-of-semisimplicial-amp}. (One piece of semi-simplicial technology from those papers is not explained here, namely \cite[Theorem 6.2]{GRW14} and its 
elaboration \cite{GRW14Err} which has been abstracted in \cite[Theorem 6.4]{BotPerl}. But these are explained in full detail in the cited papers.)

The second goal of this note is to establish basic foundational results about classifying spaces of topological categories, possibly without units (we define these in Section \ref{sec:TopCat}). The topics we will consider are: fibrancy conditions on topological categories; the effect on classifying spaces of freely adjoining units to a non-unital topological category; approximate notions of units; Quillen's Theorems A and B; the effect on classifying spaces of changing the topology on the space of objects of a topological category. In order to prove Quillen's Theorems A and B in this setting, in Section \ref{sec:Resolution} we describe a bi-semi-simplicial resolution of a semi-simplicial map induced by a functor between topological categories. We shall use our version of Quillen's Theorem B (whose formulation is a mild generalisation, due to Blumberg--Mandell \cite{BluMan}, of the usual one) in a crucial way in our 
forthcoming paper \cite{ERW17}: clarifying the details of this foundational result has been our main motivation for writing this note.

The third goal of this note is to give a proof of the Group-Completion Theorem, which plays a crucial role in \cite{GMTW}, \cite{GRW1}, and \cite{GRW14}. The formulation of this theorem which is most convenient for geometric applications is due to McDuff--Segal \cite{McDuffSeg}, but their paper elides many details. A detailed exposition of McDuff and Segal's proof has been given by Miller--Palmer \cite{MilPal}, which in combination with \cite{RWGC} proves a stronger result than the classical formulation. There are several other proofs of the group-completion theorem, due to Jardine \cites{Jardine, Goerss-Jardine},  Moerdijk \cite{Moerdijk}, and Pitsch--Scherer \cite{PitschScherer}. Our proof avoids the point-set topological subtleties of \cite{McDuffSeg}, and the model categorical subtleties of \cite{Jardine, Goerss-Jardine, Moerdijk, PitschScherer}; we think it is as elementary as possible.

Finally, we give an elementary proof that for two simplicial spaces there is a natural weak equivalence $\norm{(X \times Y)_\bullet} \simeq \norm{X_\bullet} \times \norm{Y_\bullet}$ (this can be extracted from Segal's paper \cite{Segal}). This fact has been implicitly used at some places in the literature.
\vspace{2ex}

We have attempted to make this note as self-contained as possible, and a large portion can be read with relatively little background knowledge. We assume that the reader is familiar with the language of homotopy theory and with the definition of a simplicial object and the basic examples, though we repeat the definitions. Some key results on simplicial sets, namely Lemma \ref{lem:comparison-fat-to-thin-simplicialset} and \ref{lem:RealSing} are used without proof, but in both cases there are easily accessible references. For the results of Section \ref{sec:homotopy-geoemtric-realization}, we use a fairly simple but powerful local-to-global principle for highly connected maps \cite[Theorem 6.7.9]{tomDieck}, and either Mather's first cube theorem \cite{Mather} or the Dold--Thom criterion for quasifibrations \cite{DoldThom}. In two proofs (of Theorem \ref{thm:SegalDeltaSpace} and Lemma \ref{groupcompletionproof-lemma1}) we use spectral sequences. Section \ref{sec:products} is almost elementary.

\subsection*{Acknowledgements}

J.\ Ebert was partially supported by the SFB 878. O.\ Randal-Williams was supported by EPSRC grant number EP/M027783/1.

\section{Semi-simplicial spaces}\label{sec:semisimplicialspaces}

\subsection{(Semi-)simplicial objects}

For $n \in \bN_0$, let us write $[n]$ for the linearly ordered set $ \{0 < 1 < \ldots < n\}$. Let $\Delta$ denote the category with objects the linearly ordered sets $[n]$ with $n \in \bN_0$, and with morphisms $[n] \to [m]$ the monotone functions, with composition given by composition of functions. Let $\Delta_{\inj} \subset \Delta$ denote the subcategory contains all objects, but only the \emph{injective} monotone maps. 

\begin{defn}
A \emph{simplicial object} in a category $\cC$ is a functor $\Delta^{op} \to \cC$. A \emph{semi-simplicial object} in $\cC$ is a functor $\Delta_{\inj}^{op} \to C$. We denote such a (semi-)simplicial object by $X_\bullet$, and write $X_p = X_\bullet([p])$.

A morphism of (semi-)simplicial objects is a natural transformation of functors. In this way the simplicial objects in a category $\cC$ form a category $\si \cC$, and the semi-simplicial objects form a category $\se \cC$. There is a functor $F:\si C \to \se C$, defined by restricting functors along $\Delta^{op}_{\inj} \subset \Delta^{op}$.
\end{defn}

While the description of (semi-)simplicial objects given above is convenient for certain manipulations, it is often convenient to also have a more hands-on description. The datum of a semi-simplicial object in $\cC$ is equivalent to giving a collection of objects $X_p \in \Ob(\cC)$, $p \geq 0$, together with morphisms $d_i: X_p \to X_{p-1}$ ($0 \leq i \leq p$), called \emph{face maps} which satisfy 
\[
d_i d_j = d_{j-1} d_i \,\,\text{ if }\,\, i < j.
\]
The morphism $d_i$ is associated to the unique injective monotone map $[p-1] \to [p]$ which does not hit $i$: any monotone injective map can be written as a composition of such maps, uniquely up to the identity above.

Similarly, a simplicial object in $\cC$ is given by objects $X_p \in \Ob (\cC)$, together with face maps $d_i: X_p \to X_{p-1}$ ($ 0 \leq i \leq p$) and \emph{degeneracy maps} $s_i: X_p \to X_{p+1}$ ($0 \leq i \leq p$) which satisfy the \emph{simplicial identities}
\begin{align*}
d_i d_j &= d_{j-1} d_i \,\,\text{ if }\,\, i < j,\\
s_i s_j &= s_{j+1} s_j \,\,\text{ if }\,\, i \leq  j,\\
d_i s_j &= s_{j-1} d_i \,\,\text{ if }\,\, i < j, \\
d_j s_j &= d_{j+1} s_j = \id, \\
d_i s_j &= s_j d_{i-1}\,\,\text{ if }\,\, i > j+1.
\end{align*}

In this paper, we usually think of simplicial objects as semi-simplicial objects which are equipped with additional structure, namely the degeneracy maps.

\begin{example}
The \emph{simplicial $p$-simplex} $\Delta^p_\bullet$ is the simplicial set $\Delta^p_q := \Delta ([q],[p])$. For $p=0$, one obtains $\Delta^0_q =*$, which is a terminal object in the category $\si \Set$. 

The \emph{semi-simplicial $p$-simplex} $\nabla^p_\bullet$ is the semi-simplicial set $\nabla_q^p := \Delta_{\inj} ([q],[p])$. It only has simplices in degrees $\leq p$. Note that $\nabla_q^0$ is a point when $q=0$ and empty when $q>0$.
\end{example}

\begin{defn}
An \emph{augmented semi-simplicial object} in $\cC$ is a triple $(X_{\bullet},X_{-1},\eps_\bullet)$, with $X_{-1} \in \Ob (\cC)$, $X_\bullet \in \Ob (\se \cC)$ a semi-simplicial object and morphisms $\eps_p:X_p\to X_{-1}$ such that $\eps_p \circ d_i =\eps_{p-1}$ for all $p\geq 1$ and all $0 \leq i \leq p$.
\end{defn}

Equivalently, it is a semi-simplicial object in the over-category $\cC/X_{-1}$ (see Section \ref{subsec:softunits} for a reminder of this notion).

\subsubsection*{Bi-(semi-)simplicial objects}

As the (semi-)simplicial objects in $\cC$ form a category in their own right, we may consider (semi-)simplicial objects in this category. By adjunction, this leads to the following definition.

\begin{defn}
A \emph{bi-simplicial object} in $\cC$ is a functor $X_{\bullet,\bullet}:(\Delta \times \Delta)^{op} \to \cC$ and a \emph{bi-semi-simplicial object} in $\cC$ is a functor $X_{\bullet,\bullet} : (\Delta_{\inj}\times \Delta_{\inj})^{op} \to \cC$. In either case we write $X_{p,q} = X_{\bullet, \bullet}([p], [q])$.
\end{defn}

One can think of a bi-simplicial object in $\cC$ as a simplicial object in $\si \cC$ in two ways: namely as
\[
[p] \mapsto ([q] \mapsto X_{p,q}) \quad \text{and} \quad [q] \mapsto ([p] \mapsto X_{p,q}),
\]
and similarly for bi-semi-simplicial objects. The \emph{diagonal} simplicial object $\delta X_\bullet$ is the composition of $X_{\bullet,\bullet}$ with the diagonal functor $\Delta \to \Delta \times \Delta$; similarly for bi-semi-simplicial objects. Hence $\delta X_p = X_{p,p}$.

If the category $\cC$ has finite products, we can form the \emph{exterior} product $X_\bullet \otimes Y_\bullet$ of two simplicial objects $X_\bullet, Y_\bullet \in \si \cC$; it is 
\[
(X_\bullet \otimes Y_\bullet)([p], [q]) := X([p]) \times Y([q]).
\]
The \emph{interior} product of two simplicial objects is then $X_\bullet \times Y_\bullet:=\delta (X_\bullet \otimes Y_\bullet)$, concretely 
\[
(X_\bullet \times Y_\bullet)([p]) := X([p]) \times Y([p]).
\]
Parallel notions can be defined for semi-simplicial objects, but are not very useful.

\subsubsection*{Freely adding degeneracies}

If the category $\cC$ has finite coproducts, then the forgetful functor $F: \si \cC \to \se \cC$ has a left adjoint $E$, which has the following explicit description. For a semi-simplicial object $X_\bullet \in  \se \cC$, define
\[
EX_p := \coprod_{\alpha:[p] \twoheadrightarrow [q] } X_q.
\]
Let $\beta:[r] \to [p]$ be a morphism in $\Delta$. For a surjection $\alpha: [p] \to [q]$, we factor $\alpha \circ \beta:[r] \to [q]$ as $ [r] \stackrel{\alpha'}{\twoheadrightarrow} [s] \stackrel{\beta'}{\rightarrowtail} [q]$, and define $\beta: EX_p \to EX_r$ on the summand indexed by $\alpha$ as the map $\beta'^*:X_q \to X_s \subset EX_r$.

From this adjunction, we obtain the counit $c:EFY_\bullet \to Y_\bullet$ for each $Y_\bullet \in \si \cC$, and the unit $u: X_\bullet \to F E X_\bullet$ for each $X_\bullet \in \se \cC$. Concretely, the counit is the map 
\[
EFY_p = \coprod_{\alpha:[p] \twoheadrightarrow [q] } Y_q \lra Y_p
\]
which on the summand indexed by $\alpha$ is given by $\alpha^*$. Similarly, the unit is the map 
\[
X_p \lra FEX_p = \coprod_{\alpha:[p] \twoheadrightarrow [q] } X_q
\]
which sends $X_p$ identically to the component indexed by $\id: [p]\to [p]$. Further details may be found in \cite[p. 166]{FrPic}.

\subsection{Semi-simplicial spaces and their geometric realisation}

\begin{convention}
Throughout this paper, we work in the category of compactly generated spaces as defined in \cite{Strick} (the difference to the category considered by Steenrod in \cite{Steenrod} is that we do not require the Hausdorff condition). All products of spaces are understood to be taken in the category of compactly generated spaces. One key advantage of compactly generated spaces is that taking quotients commutes with taking products in full generality, by \cite[Proposition 2.1 and 2.20]{Strick}.  Slightly abusing notation, we shall denote this category by $\Top$ call its objects topological spaces.

We think of the category $\Set$ of sets as a full subcategory of $\Top$, namely that of spaces with the discrete topology. A similar convention applies to (semi-)simplicial sets.
\end{convention}

Recall that the \emph{standard $p$-simplex} is the space
$$\Delta^p = \left\{(t_0, t_1, \ldots, t_p) \in \bR^{p+1}\, \Bigm\vert\, \sum_{i=0}^p t_i = 1 \text{ and }t_i \geq 0 \text{ for each } i \right\}.$$
To a morphism $\varphi : [p] \to [q]$ in $\Delta$ there is an associated continuous map $\varphi_* : \Delta^p \to \Delta^q$ given by $\varphi_*(t_0, t_1, \ldots, t_p) = (s_0, s_1, \ldots, s_q)$ where $s_j= \sum_{i \in \varphi^{-1}(j)} t_i$. In particular, let $d^i : \Delta^{p-1} \to \Delta^p$ be given by $(t_0, t_1, \ldots, t_p) \mapsto (t_0, t_1, \ldots, t_{i-1}, 0, t_{i}, \ldots,  t_p)$, and $s^i : \Delta^p \to \Delta^{p-1}$ be given by $(t_0, t_1, \ldots, t_p) \mapsto (t_0, t_1, \ldots, t_{i-1}, t_i + t_{i+1}, t_{i+2},  \ldots,  t_p)$.

\vspace{1ex}

The \emph{geometric realisation} of a semi-simplicial space $X_\bullet$ is the quotient space
$$\norm{X_\bullet} = \left(\coprod_p X_p \times \Delta^p\right)/\sim$$
by the equivalence relation $(x,\varphi_* t)\sim (\varphi^* x,t)$ where $\varphi$ is a morphism in $\Delta_{\inj}$. This equivalence relation is generated by the requirement that $(x,d^i t)\sim (d_i x,t)$. The \emph{$n$-skeleton} $\norm{X_\bullet}^{(n)}$ of $\norm{X_\bullet}$ is the image of $\coprod_{p=0}^n X_p \times \Delta^p$ under the quotient map. The natural map
$$\colim\limits_{n \to \infty}\norm{X_\bullet}^{(n)} \lra \norm{X_\bullet}$$
is a homeomorphism.

\begin{example}
The geometric realisation of the semi-simplicial $p$-simplex $\nabla_\bullet^p$ is the topological $p$-simplex, $\norm{\nabla_\bullet^p} \cong \Delta^p$.
\end{example}

\vspace{1ex}

The \emph{(thin) geometric realisation} of a simplicial space $X_\bullet$ is the quotient space 
$$\abs{X_\bullet} = \left(\coprod_p X_p \times \Delta^p\right)/\approx,$$
with the equivalence relation $(x,\varphi_* t)\approx (\varphi^* x,t)$ where $\varphi$ is a morphism in $\Delta$. In addition to imposing the relation $\sim$ above, the relation $\approx$ imposes $(x,s^i t)\approx (s_i x,t)$. The \emph{fat geometric realisation} of a simplicial space $X_\bullet$ is by definition $\norm{X_\bullet} := \norm{F (X_\bullet)}$, and it has a canonical map to $\abs{X_\bullet}$. Skeleta of $\abs{X_\bullet}$ are defined as above, and $\abs{X_\bullet}$ is again the colimit of its skeleta.
\begin{lem}\label{lem:comparison-fat-to-thin-simplicialset}
For each simplicial set $Y_\bullet$, the quotient map $\norm{Y_\bullet} \to \abs{Y_\bullet}$ 
is a homotopy equivalence. 
\end{lem}

The proof can be found in \cite[Proposition 2.1]{RourkeSanderson}. The following lemma allows us to compare the geometric realisation of a semi-simplicial set with the geometric realisation of the simplicial set obtained by freely adding degeneracies. Later, in Lemma \ref{lem:AdjUnitsTop}, we will prove the analogue for semi-simplicial spaces.

\begin{lem}\label{lem:AdjUnits}
For each semi-simplicial set $X_\bullet$, the map $\norm{X_\bullet} \to \norm{E X_\bullet}$ is a homotopy equivalence.
\end{lem}
\begin{proof}
We will show that the composition 
$$\norm{X_\bullet} \lra \norm{E X_\bullet} \lra \abs{E X_\bullet}$$
is a homeomorphism, whence the claim follows from Lemma \ref{lem:comparison-fat-to-thin-simplicialset}. For any simplicial set $Y_\bullet$, each point in $\abs{Y_\bullet}^{(n)} \setminus \abs{Y_\bullet}^{(n-1)}$ may be uniquely represented by a $(\sigma ; t_0, \ldots, t_n)$ with $\sigma \in Y_n$ a non-degenerate simplex. As the non-degenerate simplices of $E X_n$ are precisely given by $X_n \subset E X_n$, we may describe $\abs{EX_\bullet}^{(n)}$ as the push-out
\begin{equation*}
\xymatrix{
X_n \times \partial \Delta^n \ar[r] \ar[d] & X_n \times  \Delta^n \ar[d]\\
\abs{EX_\bullet}^{(n-1)} \ar[r] & \abs{EX_\bullet}^{(n)}.
}
\end{equation*}
Now $\norm{X_\bullet}^{(n)}$ is obtained from $\norm{X_\bullet}^{(n-1)}$ by precisely the same push-out description, which proves by induction that $\norm{X_\bullet}^{(n)} \to \abs{EX_\bullet}^{(n)}$ is a homeomorphism. Taking colimits over $n$ gives the required result.
\end{proof}

If $X_{\bullet,\bullet}$ is a bi-semi-simplicial space, we define its geometric realisation as the quotient space
\[
\norm{X_{\bullet,\bullet}} := \coprod_{p,q} X_{p,q} \times \Delta^p \times \Delta^q / \sim
\]
by the equivalence relation $((\varphi \times \psi)^*x, t, s) \sim (x, \varphi_*t, \psi_*s)$ for morphisms $\varphi \times \psi$ in $\Delta_{\inj} \times \Delta_{\inj}$. There are homeomorphisms
\begin{equation}\label{geometric-realization-bisemisimplicial}
\norm{X_{\bullet,\bullet}} \cong \norm{[p] \mapsto \norm{[q] \mapsto X_{p,q}}} \cong \norm{[q] \mapsto \norm{[p] \mapsto X_{p,q}}}
\end{equation}
and 
\begin{equation}\label{geometric-realization-bisemisimplicial2}
\norm{X_\bullet \otimes Y_\bullet} \cong \norm{X_\bullet} \times \norm{Y_\bullet},
\end{equation}
which use that we are working in the category of compactly generated spaces.

\subsubsection*{The singular simplicial set}

The \emph{singular simplicial set} of a topological space $X$ is the simplicial set with $p$-simplices $\Sing_p X:= \Top (\Delta^p,X)$, the set of continuous maps from the standard $p$-simplex to $X$, where $\varphi : [p] \to [q]$ acts via $\Top(\varphi_*, X)$. The evaluation maps
$$(\sigma, t) \mapsto \sigma(t) : \Top (\Delta^p,X) \times \Delta^p \lra X$$
assemble to a map $\abs{\Sing_\bullet X} \to X$.

\begin{lem}\label{lem:RealSing}
The maps
\begin{equation*}
\norm{\Sing_\bullet X} \stackrel{\sim}{\lra} \abs{\Sing_\bullet X} \stackrel{\sim}{\lra} X.
\end{equation*}
are weak homotopy equivalences.
\end{lem}
\begin{proof}
The first map is a weak homotopy equivalence by Lemma \ref{lem:comparison-fat-to-thin-simplicialset}. The second map is shown to be a weak equivalence in e.g.\ \cite[Theorem 16.6]{May} or \cite[Theorem 4.5.30]{FrPic}.
\end{proof}

\subsection{Extra degeneracies and semi-simplicial (null)homotopies}

If $(Y_\bullet, Y_{-1},\eps)$ is an augmented semi-simplicial space, then there is an induced map $\norm{\eps_\bullet}:\norm{Y_{\bullet}} \to Y_{-1}$. There is a standard technique for easily showing that such maps are homotopy equivalences, which goes under the name of ``having an extra degeneracy".

\begin{lem}\label{lem:ExtraDeg}
Let $(Y_\bullet, Y_{-1},\eps)$ be an augmented semi-simplicial space, and suppose there are maps $h_{p+1} : Y_p \to Y_{p+1}$ for $p \geq -1$ such that 
\begin{align*}
d_{p+1} h_{p+1} &= \mathrm{Id}_{Y_p},\\
d_i h_{p+1} &= h_p d_i \text{ for } 0 \leq i < p+1,\\
\epsilon_0 h_0 &= \mathrm{Id}_{Y_{-1}}
\end{align*}
then $\norm{\eps_\bullet}:\norm{Y_{\bullet}} \to Y_{-1}$ is a homotopy equivalence.

Dually, if there are maps $g_{p+1} : Y_p \to Y_{p+1}$ for $p \geq -1$ such that 
\begin{align*}
d_{0} g_{p+1} &= \mathrm{Id}_{Y_p},\\
d_i g_{p+1} &= g_p d_{i-1} \text{ for } 0 < i \leq p+1,\\
\epsilon_0 g_0 &= \mathrm{Id}_{Y_{-1}}
\end{align*}
then the same conclusion holds.
\end{lem}
In the first case the conditions on the maps $h_{p+1}$ are formally identical to the conditions relating face maps $d_i$ to degeneracy maps $s_i$, except that $h_{p+1}$ behaves like a hypothetical degeneracy map $s_{p+1}$, whereas in the definition of a simplicial object there are only degeneracy maps $s_0, s_1, \ldots, s_p : Y_p \to Y_{p+1}$. For this reason such a collection of maps $h_{p+1}$ is often called an \emph{extra degeneracy}. (Similarly, $g_{p+1}$ behaves like a hypothetical degeneracy map $s_{-1} : Y_p \to Y_{p+1}$.)

\begin{proof}
Let us just consider the first case. We have $h_0 : Y_{-1} \to Y_0 \subset \norm{Y_\bullet}$ and $\norm{\epsilon_\bullet} \circ h_0 = \mathrm{Id}_{Y_{-1}}$. The maps $[0,1] \times Y_p \times \Delta^p \to Y_{p+1} \times \Delta^{p+1} \to \norm{Y_{\bullet}}$, defined by
\[
 (s;x;t_0,\ldots,t_p) \mapsto (h_{p+1} (x); (1-s)t_0, \ldots, (1-s)t_p,s),
\]
respect the equivalence relation used in the definition of the geometric realisation. Since taking products and taking quotients commutes in compactly generated spaces, this yields a homotopy
$ H: [0,1] \times \norm{Y_\bullet} \to \norm{Y_\bullet}$, and one verifies that $H(0,-)= \mathrm{Id}_{\norm{Y_\bullet}}$ and that $H(1, -) = h_0 \circ \norm{\epsilon_\bullet}$. 
\end{proof}

Any semi-simplicial space $Y_\bullet$ is augmented over a point $*$ in a unique way. The data of an extra degeneracy in this case gives in particular a point $y_0 : * \to Y_0$, and the homotopy in the proof gives a contraction of $\norm{Y_\bullet}$ to the point $\{y_0\} \subset Y_0 \subset \norm{Y_\bullet}$. This can be generalised to maps of semi-simplicial spaces, as follows.

\begin{lem}\label{lem:semisimplicialcontraction}
Let $f_\bullet: X_\bullet \to Y_\bullet$ be a map of semi-simplicial spaces and $y_0 \in Y_0$. A \emph{semi-simplicial nullhomotopy} from $f_\bullet$ to $y_0$ is a collection of continuous maps $h_{p+1}: X_p \to Y_{p+1}$ such that
\begin{align*}
d_{p+1} h_{p+1} &= f_p,\\
 d_i h_{p+1} &= h_p d_i \,\,\, \text{ for }\,0 \leq i \leq p \,\text{ and } p \geq 1,\\
   d_0 h_{1} &\equiv y_0.
\end{align*}
Such a semi-simplicial nullhomotopy induces a homotopy from $\norm{f_{\bullet}}$ to the constant map $\norm{X_\bullet} \to \{ y_0\} \subset Y_0 \subset \norm{Y_\bullet}$.
\end{lem}

\begin{proof}
Use the same formula as in the proof of Lemma \ref{lem:ExtraDeg} to obtain a homotopy  $H: [0,1] \times \norm{X_\bullet} \to \norm{Y_\bullet}$ with $H(0,- )= \norm{f_\bullet}$ and $H(1, -)$ the constant map with value $y_0$.  
\end{proof}

\begin{example}\label{ex:realization-of-simplicial-simplex}
The fat geometric realisation of the simplicial $n$-simplex $\Delta_\bullet^n$ is contractible (it is not homeomorphic to $\Delta^n$).
Recall that $\Delta^n_p = \Delta([p],[n])$ and let $h_{p+1} : \Delta([p],[n]) \to \Delta([p+1],[n])$ be the map that sends $\eta: [p] \to [n]$ to the map $\eta': [p+1]\to [n]$ which is defined by $\eta' (i)=\eta(i)$ for $i \leq p$ and $\eta' (p+1):= n$. 
This is a simplicial nullhomotopy from $\id_{\Delta_\bullet^n}$ to the vertex $n \in \Delta_0^n$, and hence the claim follows from Lemma \ref{lem:semisimplicialcontraction}.
\end{example}

More generally, we have the notion of a semi-simplicial homotopy between semi-simplicial maps.

\begin{lem}\label{lem:semisimplicialhomotopy}
Let $f_\bullet, g_\bullet: X_\bullet \to Y_\bullet$ be maps of semi-simplicial spaces. A \emph{semi-simplicial homotopy} from $f_\bullet$ to $g_\bullet$ is a collection of continuous maps $h_{p+1, i}: X_p \to Y_{p+1}$ for $i=0,1,\ldots, p$ such that
\begin{align*}
d_i h_{p+1, i} &= d_i h_{p+1, i-1} \,\,\, \text{ for }\,0 < i \leq p,\\
d_i h_{p+1, j} &= h_{p, j-1} d_i \,\,\, \text{ for }\,0 \leq i < j,\\
d_i h_{p+1, j} &= h_{p, j} d_i \,\,\, \text{ for }\, j+1 < i \leq p,\\
d_0 h_{p+1, 0} &= f_p,\\
d_{p+1} h_{p+1, p} &= g_p.
\end{align*}
Such a semi-simplicial homotopy induces a homotopy from $\norm{f_{\bullet}}$ to $\norm{g_\bullet}$.
\end{lem}

\begin{proof}
Consider the maps
\begin{align*}
\psi_{p+1,i} : \Delta^{p+1} &\lra \Delta^1 \times \Delta^p\\
\sum_{j=0}^{p+1} t_j e_j &\longmapsto \sum_{j=0}^i t_j (0, e_j) + \sum_{j=i}^{p} t_{j+1} (1, e_j)
\end{align*}
for $i=0,1,\ldots,p$, giving the standard decomposition of the prism into simplices. The maps 
\begin{align*}
\psi_{p+1,i}(\Delta^{p+1}) \times X_p &\lra \Delta^{p+1} \times Y_{p+1} \subset \norm{Y_\bullet}\\
(\psi_{p+1,i}(t), x) &\longmapsto (t, h_{p+1,i}(x))
\end{align*}
glue to maps $\phi_p : [0,1] \times \Delta^p \times X_p \to \norm{Y_\bullet}$ (using the first set of identities) which in turn glue to a map $\phi : [0,1] \times \norm{X_\bullet} \to \norm{Y_\bullet}$ (using the second and third set of identities). This gives the required homotopy (using the fourth and fifth set of identities).
\end{proof}

\subsection{Spectral sequences}\label{subsec:spectralsequence}

The space $\norm{X_\bullet}$ is filtered by its skeleta $\norm{X_{\bullet}}^{(n)}$, where $\norm{X_{\bullet}}^{(0)}=X_0$, and 
\begin{equation}\label{filtration-as-pushout}
 \norm{X_{\bullet}}^{(n)} = \norm{X_{\bullet}}^{(n-1)} \cup_{X_n \times \partial \Delta^n} X_n \times \Delta^n.
\end{equation}
This filtration has the property that each map $K \to \norm{X_{\bullet}}$ from a compact Hausdorff space $K$ factors through some finite stage; see e.g.\ \cite[Proposition A.1]{Hatcher} for a related argument, or \cite[Lemma 3.6]{Strick} for a general argument.

Recall that a \emph{local coefficient system} on a space $Y$ is a functor $\cL$ from the fundamental groupoid $\Pi_1 (Y)$ to the category of $R$-modules for a commutative ring $R$. If $Y$ is semi-locally simply-connected then we may also consider a local coefficient system on $Y$ to be a bundle $\cL \to Y$ of $R$-modules.

For any system of local coefficients $\cL$ on $\norm{X_\bullet}$, the skeletal filtration yields a spectral sequence 
$$E^1_{p,q} = H_{p+q}(\norm{X_{\bullet}}^{(q)},\norm{X_{\bullet}}^{(q-1)};\cL) \Longrightarrow H_{p+q}(\norm{X_{\bullet}};\cL),$$
which is strongly convergent as each map from a simplex to $\norm{X_\bullet}$ lands in some finite skeleton. Let $\cL\vert_{X_q \times \Delta^q}$ be the pullback of $\cL$ along $X_q \times \Delta^q \to \norm{X_\bullet}$, and $\cL_q$ be the restriction of $\cL\vert_{X_q \times \Delta^q}$ to $X_q \cong X_q \times b_q$ where $b_q \in \Delta^q$ is the barycentre. The natural map
$$H_{p+q}(X_q \times \Delta^q, X_q \times \partial \Delta^q;\cL\vert_{X_q \times \Delta^q}) \lra H_{p+q}(\norm{X_{\bullet}}^{(q)},\norm{X_{\bullet}}^{(q-1)};\cL)$$
is an isomorphism, using the description \eqref{filtration-as-pushout} and excision. The contraction of $\Delta^q$ to $b_q \in \Delta^q$ determines an isomorphism $\cL\vert_{X_q \times \Delta^q} \cong \pi_1^*\cL_q$, and the K{\"u}nneth map 
$$H_p(X_q ; \cL_q) \cong H_p(X_q ; \cL_q) \otimes H_q(\Delta^q, \partial \Delta^q ; \bZ) \lra H_{p+q}(X_q \times \Delta^q, X_q \times \partial \Delta^q;\pi_1^*\cL_q)$$
is an isomorphism (as the homology of $(\Delta^q, \partial \Delta^q)$ is free). Thus we obtain the description
\begin{equation*}
E^1_{p,q} \cong H_{p}(X_q ; \cL_q) \Longrightarrow H_{p+q}(\norm{X_{\bullet}};\cL)
\end{equation*}
for this spectral sequence. To each face map $d_i : X_q \to X_{q-1}$ there is a unique homotopy class of path in $\Delta^q$ from $d^i(b_{q-1})$ to $b_q$, monodromy along which gives a preferred map of local coefficient systems $\phi_i : \cL_q \to \cL_{q-1}$ covering $d_i$. One may show (see \cite[\S 5]{Segal-classifying}) that the $d^1$-differential is
$$d^1 = \sum_{i=0}^q (-1)^i (d_i, \phi_i)_* : H_{p}(X_q ; \cL_q) \lra H_{p}(X_{q-1}; \cL_{q-1})$$
the alternating sum of the maps induced on homology by the face maps.

More generally, if $(X_\bullet,X_{-1},\eps)$ is an augmented semi-simplicial space then (replacing $X_{-1}$ by the mapping cylinder of $\norm{\epsilon_\bullet} : \norm{ X_\bullet} \to X_{-1}$ and) setting $F_{-1} = (X_{-1}, X_{-1})$ and $F_q = (X_{-1}, \norm{X_{\bullet}}^{(q)})$ for $q \geq 0$ gives a filtration of pairs, and hence for each local coefficient system $\cL$ on $X_{-1}$ a spectral sequence with $E^1_{p,q} \cong H_{p}(X_q;\cL_q)$ for $p \geq 0$ and $q \geq -1$, which converges to $H_{p+q+1}(X_{-1}, \norm{X_{\bullet}};\cL)$.

\section{Results on the homotopy type of the geometric realisation}\label{sec:homotopy-geoemtric-realization}

In this section we shall collect results which allow one to deduce homotopical statements about geometric realisation of a map $f_\bullet : X_\bullet \to Y_\bullet$ of semi-simplicial spaces from homotopical statements about the maps $f_p : X_p \to Y_p$. One says that a semi-simplicial map $f_\bullet$ has a certain property \emph{levelwise} if each map $f_p$ has that property. As a basic technical tool for gluing together $k$-connected maps, we will take Theorem 6.7.9 of tom Dieck's book \cite{tomDieck}.

\begin{lem}\label{lem:inclusion-skeleta}
For $m \geq n$ the inclusion $\norm{X_{\bullet}}^{(n)} \to  \norm{X_{\bullet}}^{(m)}$ is $n$-connected, and the inclusion $\norm{X_{\bullet}}^{(n)} \to \norm{X_{\bullet}}$ is $n$-connected.
\end{lem}

\begin{proof}
For the first claim, it is enough to prove that the inclusion $\norm{X_{\bullet}}^{(n)} \to \norm{X_{\bullet}}^{(n+1)}$ is $n$-connected. To see this, let $b \in \Delta^{n+1}$ be the barycentre and consider the covering of $\norm{X_{\bullet}}^{(n+1)}$ by the open sets 
\begin{align*}
U_0^X &=\norm{X_{\bullet}}^{(n+1)} \setminus (X_{n+1} \times \{b\})\simeq \norm{X_{\bullet}}^{(n)}\\
U_1^X &=X_{n+1} \times \inter{\Delta^{n+1}}\simeq X_{n+1}
\end{align*}
with intersection $U_{0}^X \cap U_{1}^X \simeq X_{n+1} \times \partial \Delta^{n+1}$. Applying \cite[Theorem 6.7.9]{tomDieck} to the map
$$(U_0^X, U_0^X, U_1^X \cap U_0^X) \lra (\norm{X_\bullet}^{(n+1)}, U_0^X, U_1^X)$$
shows that $\norm{X_\bullet}^{(n)} \overset{\sim}\to U_0^X \to \norm{X_\bullet}^{(n+1)}$ is $n$-connected, as required. The second claim follows from the first one and the fact that a map from a compact Hausdorff space to $\norm{X_\bullet}$ factors through a skeleton.
\end{proof}

\begin{thm}\label{thm:levelwiseequivalence}
Let $f_\bullet:X_\bullet\to Y_\bullet$ be a map of semi-simplicial spaces which is a levelwise weak homotopy equivalence. Then $\norm{f_\bullet}:\norm{X_\bullet}\to \norm{Y_\bullet}$ is a weak homotopy equivalence.
\end{thm}

\begin{proof}
By Lemma \ref{lem:inclusion-skeleta}, it is enough to show that $\norm{f_\bullet}:\norm{X_\bullet}^{(n)}\to \norm{Y_\bullet}^{(n)}$ is a weak equivalence for each $n$, and this may be shown by induction on $n$. The case $n=0$ is trivial. For the induction step, consider the open sets $U_0^X, U_1^X \subset \norm{X_\bullet}^{(n+1)}$ from the proof of Lemma \ref{lem:inclusion-skeleta} and the analogous $U_0^Y, U_1^Y \subset \norm{Y_\bullet}^{(n+1)}$. By induction hypothesis, the restriction of $\norm{f_{\bullet}}$ to $U_0^X\to U_0^Y$ is a weak equivalence, and so is the restriction $U_1^X\to U_1^Y$ and $U_{0}^X \cap U_1^X \to U_{0}^Y \cap U_1^Y$. The inductive step then follows using \cite[Theorem 6.7.9]{tomDieck}.
\end{proof}

\begin{remark}
Theorem \ref{thm:levelwiseequivalence} is \emph{false} in general for the thin geometric realisation of simplicial spaces. This is the main reason why---even for simplicial spaces---it is often preferable to consider the fat geometric realisation. A concrete counterexample was given by Lawson in response to a question on MathOverflow \cite{Tyler}.
\end{remark}

Theorem \ref{thm:levelwiseequivalence} has the following useful generalisation.

\begin{lem}\label{lem:connecttivity-of-semisimplicial-amp}
Let $f_\bullet: X_\bullet \to Y_\bullet$ be a map of semi-simplicial spaces. If $f_p:X_p \to Y_p$ is $(k-p)$-connected for all $p$, then $\norm{f_\bullet}$ is $k$-connected.
\end{lem}

\begin{proof}
By Lemma \ref{lem:inclusion-skeleta} it is enough to show that $\norm{f_\bullet}^{(n)} : \norm{X_\bullet}^{(n)} \to \norm{Y_\bullet}^{(n)}$ is $k$-connected for each $n$. The case $n=0$ is trivial. For the induction step, we may as well suppose that $X_i = Y_i = \emptyset$ for $i > n$ and that $\norm{f_\bullet}^{(n-1)}$ is $k$-connected. We factorise $f_\bullet$ as 
\begin{equation}\label{eq:lem:connecttivity-of-semisimplicial-amp-proof}
X_\bullet \stackrel{j_\bullet}{\lra} W_\bullet \stackrel{g_\bullet}{\lra} Z_\bullet \stackrel{h_\bullet}{\lra} Y_\bullet
\end{equation}
as follows. The semi-simplicial space $W_\bullet$ has $W_i = Y_i$ for $i<n$, $W_n =X_n$ and $W_i = \emptyset$ for $i>n$. The face maps $W_n \to W_{n-1}$ are the compositions $f_{n-1} \circ d_i = d_i \circ f_n$, and the other face maps are the same as those for $X_\bullet$. The map $j_n$ is the identity, and $j_i = f_i$ for $i< n$. 

Then factorise $f_n$ as
$$f_n : X_n \overset{g_n}\lra Z_n \overset{h_n}\lra Y_n$$
where $h_n$ is a weak homotopy equivalence, and $Z_n$ is obtained from $X_n$ by attaching cells of dimension at least $(k-n+1)$. For $i < n$ let $Z_i = Y_i$, and for $i > n$ let $Z_i = \emptyset$. The map $g_i$ is the identity for $i<n$, and $h_i: Z_i \to Y_i$ is the identity as well. This yields the factorisation \eqref{eq:lem:connecttivity-of-semisimplicial-amp-proof}. 

The map $h_\bullet$ is a levelwise weak equivalence, and so $\norm{h_\bullet}$ is a weak equivalence by Theorem \ref{thm:levelwiseequivalence}. Moreover, $\norm{W_\bullet}^{(n-1)} = \norm{Z_\bullet}^{(n-1)}$ and the pair $(Z_n \times \Delta^n, Z_n \times \partial \Delta^n)$ is obtained from the pair $(W_n \times \Delta^n, W_n \times \partial \Delta^n)$ by attaching cells of dimension at least $(k+1)$, so $\norm{Z_\bullet}^{(n)}$ is obtained from $\norm{W_\bullet}^{(n)}$ by attaching cells of dimension at least $(k+1)$: in particular, $\norm{g_\bullet}:\norm{W_\bullet}^{(n)} \to \norm{Z_\bullet}^{(n)}$ is $k$-connected.
By the inductive hypothesis, $\norm{j_\bullet}^{(n-1)}: \norm{X_\bullet}^{(n-1)} \to \norm{W_\bullet}^{(n-1)}$ is $k$-connected and $j_n$ is the identity. Using the notation introduced in the proof of Theorem \ref{thm:levelwiseequivalence}, we get that $U_0^X \to U_0^W$ is $k$-connected, while $U_1^X \to U_1^W$ and $U_{0}^X \cap U_{1}^X\to U_{0}^W \cap U_{1}^W$ are weak equivalences. From \cite[Theorem 6.7.9]{tomDieck}, it follows that $\norm{j_\bullet}$ is $k$-connected. 
\end{proof}

Using this we can now prove the analogue of Lemma \ref{lem:AdjUnits} for semi-simplicial spaces, rather than semi-simplicial sets.

\begin{lem}\label{lem:AdjUnitsTop}
For each semi-simplicial space $X_\bullet$, the map $\norm{X_\bullet} \to \norm{E X_\bullet}$ is a weak homotopy equivalence.

\end{lem}
\begin{proof}
Consider the bi-semi-simplicial set $\Sing_p(X_q)$, with $\Sing_p(EX_\bullet) = E (\Sing_p X_\bullet)$ so giving a commutative square
\begin{equation*}
\xymatrix{
\norm{\Sing_\bullet X_\bullet} \ar[d] \ar[r]& \norm{E (\Sing_\bullet X_\bullet)} \ar@{=}[r]& \norm{\Sing_\bullet(EX_\bullet)} \ar[d]\\
\norm{X_\bullet} \ar[rr]&& \norm{EX_\bullet}.
}
\end{equation*}
The vertical maps are weak equivalences by Lemma \ref{lem:RealSing} and Theorem \ref{thm:levelwiseequivalence}, and the top map is a weak equivalence by Lemma \ref{lem:AdjUnits} and Theorem \ref{thm:levelwiseequivalence}; hence the bottom map is a weak equivalence. 
\end{proof}

\begin{defn}
A commutative square 
\[
 \xymatrix{X_1 \ar[d]^{f} \ar[r]^{k_1} & Y_1 \ar[d]^{g}\\
 X_0 \ar[r]^{k_0} & Y_0
 }
\]
is called \emph{homotopy cartesian} if for each basepoint $x \in X_0$, the map $\hofib_x (f)  \to \hofib_{k_0(x)} (g)$, induced by $k_0$ and $k_1$, is a weak homotopy equivalence. 
\end{defn}

\begin{rem}\label{remark.hocartesianness-symmetric}
Equivalently, one can express this condition by saying that for all $y \in Y_1$, the induced map $\hofib_y (k_1) \to \hofib_{g(y)} (k_0)$ is a weak homotopy equivalence. 

More symmetrically, one can express this condition by saying that the canonical map from $X_1$ to the \emph{homotopy fibre product}
$$X_0 \times_{Y_0}^h Y_1 := \{(x_0, y_1, \gamma) \in X_0 \times Y_1 \times \mathrm{map}([0,1], Y_0) \, \vert \, \gamma(0) = k_0(x_0), \gamma(1) = g(y_1)\}$$
is a weak homotopy equivalence. 
\end{rem}

Let us record the 2-out-of-3 properties enjoyed by homotopy cartesian squares. Given adjacent commutative squares
\begin{equation*}
\xymatrix{
X_1 \ar[r]^{k_1} \ar[d]^f& Y_1 \ar[r]^{l_1} \ar[d]^g & Z_1 \ar[d]^h\\
X_0 \ar[r]^{k_0}& Y_0  \ar[r]^{l_0} & Z_0
}
\end{equation*}
then
\begin{enumerate}[(i)]
\item if the left and right squares are homotopy cartesian, the outer square is homotopy cartesian;

\item is the right and outer squares are homotopy cartesian, the left square is homotopy cartesian;

\item if the left and outer squares are homotopy cartesian, and $k_0$ is 0-connected, the right square is homotopy cartesian.
\end{enumerate}

\begin{defn}\label{defn:hocartesian.morpism}
A map $f_\bullet: X_\bullet \to Y_\bullet$ of semi-simplicial spaces is called \emph{homotopy cartesian} if for each $p \geq 1 $ and each $0 \leq i \leq p$, the square 
\begin{equation}\label{diag:hocartesian.morpism}
\begin{gathered}
\xymatrix{
X_p \ar[d]^{f_p} \ar[r]^{d_i}  & X_{p-1} \ar[d]^{f_{p-1}}\\
Y_p \ar[r] \ar[r]^{d_i} & Y_{p-1}
}
\end{gathered}
\end{equation}
is homotopy cartesian.
\end{defn}

For each $p$ there are $p+1$ conditions to be checked. The next lemma shows that the number of conditions to be checked can be drastically reduced.

\begin{lem}\label{lem:minimal-checking-for-hocartesianness}
To prove that $f_\bullet: X_\bullet \to Y_\bullet$ is homotopy cartesian, it is enough to verify that \eqref{diag:hocartesian.morpism} is homotopy cartesian for those $(p,i)$ with $i=0$ and for $(p,i)=(1,1)$. Dually, it is enough to verify that \eqref{diag:hocartesian.morpism} is homotopy cartesian for those $(p,i)$ with $i=p$ and for $(p,i)=(1,0)$.
\end{lem}

\begin{proof}
We treat only the first case. Consider the commutative cube
\begin{equation*}
\xymatrix{
 & X_k \ar'[d]^-{f_k}[dd] \ar[dl]_-{d_k} \ar[rr]^-{d_0^{k-1}} & & X_1 \ar[dd]^-{f_1} \ar[dl]^-{d_1}\\
X_{k-1} \ar[dd]^-{f_{k-1}} \ar[rr]_-{d_0^{k-1}} & & X_0 \ar[dd]^-{f_0}\\
 & Y_k \ar[dl]_-{d_k} \ar'[r]^-{d_0^{k-1}}[rr] & & Y_1 \ar[dl]^-{d_1}\\
Y_{k-1} \ar[rr]^-{d_0^{k-1}} & & Y_0.
}
\end{equation*}
By hypothesis the front, back, and right faces are homotopy cartesian, so the left face is too. But each structure map $X_p \to X_0$ can be written as the composition of maps of the form $d_0$ and $d_k : X_k \to X_{k-1}$. Therefore, for each $\eta: [0] \to [p]$, the square
\[
\xymatrix{
X_p \ar[d]^{f_p} \ar[r]^{\eta^* }  & X_{0} \ar[d]^{f_{0}}\\
Y_p \ar[r] \ar[r]^{\eta^*} & Y_{0}
}
\]
is homotopy cartesian. The result then follows easily. 
\end{proof}

The following is due to Segal \cite[Proposition 1.6]{Segal}.

\begin{thm}\label{lem:simplicial-hocartesianness}
Let $f: X_{\bullet} \to Y_{\bullet}$ be a homotopy cartesian map of semi-simplicial spaces. Then the square
\[
\xymatrix{
X_0 \ar[d]^{f_0} \ar[r] & \norm{X_{\bullet}}\ar[d]^{\norm{f_{\bullet}}}\\
Y_0 \ar[r] & \norm{Y_{\bullet}}
}
\]
is also homotopy cartesian.
\end{thm}

\begin{proof}[First proof]
We prove the result by induction on skeleta. There are commutative cubes
\[
 \xymatrix{
& & X_p \times \{v_0\} \ar'[d]'[dd][ddd]^-{f_p \times \{v_0\}}\ar[ld] \ar[rrr]^{d_0^p} &&& X_0 \ar[ld] \ar[ddd]^-{f_0}\\
& X_p \times \partial \Delta^p \ar[ld] \ar[rrr] \ar'[d][ddd]^-{f_p \times \partial \Delta^p} & & & \norm{X_\bullet}^{(p-1)} \ar[ddd]^-{\norm{f_\bullet}^{(p-1)}} \ar[ld]\\
X_p \times \Delta^p  \ar[rrr] \ar[ddd]^-{f_p \times \Delta^p} & & & \norm{X_\bullet}^{(p)} \ar[ddd]^-{\norm{f_\bullet}^{(p)}}\\
& & Y_p \times \{v_0\} \ar[ld] \ar'[r]'[rr]^-{d_0^p}[rrr] &&& Y_0 \ar[ld]\\
& Y_p \times \partial \Delta^p \ar[ld]\ar'[rr][rrr] & & & \norm{Y_\bullet}^{(p-1)} \ar[ld]\\
Y_p \times \Delta^p  \ar[rrr] & & & \norm{Y_\bullet}^{(p)}.
 }
\]

Consider first the back cube. If $p=1$ then the front face is homotopy cartesian by hypothesis. If $p>1$ then the right-hand face is homotopy cartesian by inductive assumption, the left-hand face is homotopy cartesian, the back face is homotopy cartesian by hypothesis, and $Y_p \times \{v_0\} \to Y_p \times \partial \Delta^p$ is 0-connected: thus by the 2-out-of-3 property of homotopy cartesian squares the front face of the back cube is homotopy cartesian.

Consider now the front cube. The left-hand face is homotopy cartesian and by the above the back face is too. The top and bottom faces are homotopy co-cartesian, so this cube satisfies the hypotheses of Mather's first cube theorem \cite{Mather}. Thus the right-hand face of the front cube is homotopy cartesian, and hence the right-hand face of the outer cube is also homotopy cartesian, as required.
\end{proof}
\begin{proof}[Second proof]
First consider the case where each $f_p$ is a fibration. In this case, the lemma follows from the fact that the geometric realisation $\norm{f_\bullet}:\norm{X_\bullet} \to \norm{Y_\bullet}$ is a quasifibration, which in turn follows from applying the Dold--Thom criterion \cite[Satz 2.2, Hilfssatz 2.10 and Satz 2.12]{DoldThom} (a convenient reference is \cite[Lemma 4.K.3]{Hatcher}).

In the general case, we factor $f_p$ functorially as a composition $X_p \stackrel{h_p}{\to} Z_p \stackrel{g_p}{\to} Y_p$ with a weak equivalence $h_p$ and a fibration $g_p$. Then $Z_\bullet$ is a semi-simplicial space, and $h_\bullet$, $g_\bullet$ are semi-simplicial maps. 
In the diagram
\[
\xymatrix{
X_p \ar[d]^{d_i} \ar[r]^{h_p} & Z_p \ar[d]^{d_i}\ar[r]^{g_p} & Y_p \ar[d]^{d_i}\\
X_{p-1} \ar[r]^{h_{p-1}} & Z_{p-1} \ar[r]^{g_{p-1}} & Y_{p-1},
}
\]
the maps $h_p$ and $h_{p-1}$ are weak homotopy equivalences, and it follows that the right square is homotopy cartesian.  The lower square in 
\[
\xymatrix{
X_0 \ar[r] \ar[d]^{h_0}& \norm{X_\bullet}\ar[d]^{\norm{h_\bullet}}\\
Z_0 \ar[r] \ar[d]^{g_0}& \norm{Z_\bullet}\ar[d]^{\norm{g_\bullet}}\\
Y_0 \ar[r] & \norm{Y_\bullet}
}
\]
is homotopy cartesian by the first part of the proof, and the upper square is homotopy cartesian as $h_0$ and $\norm{h_\bullet}$ are both weak equivalences, by Theorem \ref{thm:levelwiseequivalence}.
\end{proof}

\begin{lem}\label{lem:augmented-cartesian}
Let $\eps:X_\bullet\to X_{-1}$ and $\eps:Y_\bullet\to Y_{-1}$ be augmented semi-simplicial spaces and let $(f_\bullet,f):(X_{\bullet},X_{-1})\to (Y_{\bullet},Y_{-1})$ be a map of augmented semi-simplicial spaces. If for each $p\geq 0$ the square
\[
\xymatrix{
X_p \ar[r]^{f_p} \ar[d]^{\eps_p} & Y_p \ar[d]^{\eps_p}\\
X_{-1} \ar[r]^{f} & Y_{-1}
}
\]
is homotopy cartesian, then so is the square
\[
\xymatrix{
\norm{X_\bullet} \ar[r]^{\norm{f_\bullet}} \ar[d]^{\norm{\eps_\bullet}} & \norm{Y_\bullet} \ar[d]^{\norm{\eps_\bullet}}\\
X_{-1} \ar[r]^{f} & Y_{-1}.
}
\]
\end{lem}

\begin{proof}
The diagram
\begin{equation*}
\xymatrix{
X_p \ar[d]^{f_p} \ar[r]^{d_i } & X_{p-1}\ar[d]^{f_{p-1}} \ar[r]^{\eps_{p-1}} & X_{-1}\ar[d]^{f}\\
Y_p \ar[r]^{d_i} & Y_{p-1} \ar[r]^{\eps_{p-1}} & Y_{-1}
}
\end{equation*}
has right-hand and outer squares homotopy cartesian by hypothesis, so the left-hand square is also homotopy cartesian. Thus the map $f_\bullet$ is homotopy cartesian and we can apply Theorem \ref{lem:simplicial-hocartesianness}, which shows that the left-hand square in
\[
\xymatrix{
X_0 \ar[r]^{\iota} \ar[d]^{f_0} & \norm{X_\bullet} \ar[d]^{\norm{f_\bullet}} \ar[r]^{\norm{\eps_\bullet}} & X_{-1} \ar[d]^{f}\\
Y_0 \ar[r]^{\iota} & \norm{Y_\bullet} \ar[r]^{\norm{\eps_\bullet}} & Y_{-1}
}
\]
is homotopy cartesian. As $\norm{\eps_\bullet} \circ \iota = \eps_0$ the outer square is homotopy cartesian by hypothesis, and $\iota: Y_0 \to \norm{Y_\bullet}$ is $0$-connected, so the right-hand square is also homotopy cartesian as required.
\end{proof}

\begin{lem}\label{lem:augmented-fibration}
Let $\eps_\bullet:X_\bullet\to X_{-1}$ be an augmented semi-simplicial space such that each $\eps_p:X_p \to X_{-1}$ is a quasifibration. Then for each $x \in X_{-1}$, the natural map
\[
\norm{\eps^{-1}_\bullet (x)} \lra \hofib_{x} \norm{\eps_\bullet}
\]
is a weak homotopy equivalence.
\end{lem}
\begin{proof}
The diagrams
\[
\xymatrix{
\eps_p^{-1}(x) \ar[d] \ar[r] & X_p \ar[d]^{\eps_p}\\
\{x\} \ar[r] & X_{-1}
}
\]
form a map of augmented semi-simplicial spaces, and by assumption this map is homotopy cartesian. The statement then follows from Lemma \ref{lem:augmented-cartesian}.
\end{proof}

\begin{cor}\label{cor:constant-simplicial-space}
Let $X$ be a topological space and consider the constant semi-simplicial space $X_\bullet$ ($X_p:=X$ and all face maps are the identity). Then the inclusion $\iota:X =\norm{X_\bullet}^{(0)} \to \norm{X_\bullet}$ is a weak equivalence. 
\end{cor}

\begin{proof}
The identity map(s) $X_p \to X$ form an augmentation $\eps_\bullet:X_\bullet \to X$ and the composition $\norm{\epsilon_\bullet} \circ \iota$ is the identity. The semi-simplicial space $\eps^{-1}_\bullet (x)$ is the terminal semi-simplicial space and hence has contractible geometric realisation. It then follows from Lemma \ref{lem:augmented-fibration} that $\norm{\eps_\bullet}$ is a weak homotopy equivalence, whence the claim follows.
\end{proof}

The following result is due to Segal \cite[Proposition 1.5]{Segal}, though we have generalised the formulation a little. It plays a key role in his theory of $\Gamma$-spaces (which will be used in e.g.\ \cite{ERW17}), and is also a key ingredient in \cite{GRW1}.

\begin{thm}\label{thm:SegalDeltaSpace}
Let $X_\bullet$ be a semi-simplicial space and assume that 
\begin{enumerate}[(i)]
\item $X_0 \simeq *$.
\item The map $\kappa_p:X_p \to (X_1)^p$ given by $(\iota_1^*, \ldots, \iota_p^*)$, where $\iota_j:[1] \to [p]$ is the map $0 \mapsto j-1$, $1 \mapsto j$, is a weak homotopy equivalence.
\end{enumerate}
Then 
\begin{enumerate}[(i)]
\setcounter{enumi}{2}
\item If $X_1$ is $k$-connected, then $\norm{X_\bullet}$ is $(k+1)$-connected. 
\item If the squares
\[
\xymatrix{
X_2 \ar[r]^{d_1} \ar[d]^{d_2} & X_1 \ar[d]^{d_1} & \ar@{}[d]|{\mbox{and}} & X_2 \ar[r]^{d_1} \ar[d]^{d_0} & X_1 \ar[d]^{d_0}\\
X_1 \ar[r]^{d_1} & X_0 & & X_1 \ar[r]^{d_1} & X_0
}
\]
are homotopy cartesian and $X_1 \neq \emptyset$, then the tautological map 
\[
X_1 \to \tilde{\Omega} \norm{X_\bullet}
\]
(the target is the space of paths that begin and end in the contractible subspace $X_0 \subset \norm{X_\bullet}$) is a weak homotopy equivalence.
\end{enumerate}
\end{thm}

If $X_\bullet$ is simplicial, and not just semi-simplicial, then the statement of this theorem and its proof can be simplified, which we shall explain in Remark \ref{rem:tothm:segal-deltaspace} below.

Under assumption (ii) we can form the morphism
$$\mu: X_1 \times X_1 \overset{d_0 \times d_2}{\underset{\simeq}{\longleftarrow}} X_2 \overset{d_1}\lra X_1$$
in the homotopy category, which makes $X_1$ into a non-unital homotopy associative $H$-space. Assumptions (i) and (ii) should be thought of as saying that $X_\bullet$ is a model for the nerve of this $H$-space. 

The assumption in (iv) can be expressed, by taking vertical homotopy fibres, as asking that for each $x \in X_1$ the maps $\mu(x, -), \mu(-, x) : X_1 \to X_1$ be weak homotopy equivalences. That is, it models the $H$-space $X_1$ being \emph{grouplike}. In particular, $\mu$ induces an associative product $- \cdot - : \pi_0(X_1) \times \pi_0(X_1) \to \pi_0(X_1)$ for which $[x] \cdot -, - \cdot [x] :\pi_0(X_1) \to \pi_0(X_1)$ are bijections for all $[x]$. As $X_1 \neq \emptyset$ we may choose an $[x] \in \pi_0(X_1)$, for which there is a unique $[e] \in \pi_0(X_1)$ such that $[x] \cdot [e] = [x]$. But then for any $y$ we have
$$[x] \cdot [y] = ([x] \cdot [e]) \cdot [y] = [x] \cdot ([e] \cdot [y]),$$
so $[e] \cdot [y] = [y]$ for any $[y]$ as $[x] \cdot -$ is injective. But then
$$([y] \cdot [e]) \cdot [y] = [y] \cdot ([e] \cdot [y]) = [y] \cdot [y]$$
and so $[y] \cdot [e] = [y]$ as $- \cdot [y]$ is injective. Hence $[e]$ is an two-sided identity element for $-\cdot-$, making $(\pi_0(X_1), \cdot, [e])$ an associative unital monoid. As each $[y] \cdot -$ is a bijection, it is easy to see that this is in fact a group. One consequence is that the map $\mu(e, -) : X_1 \to X_1$ satisfies $\mu(e, \mu(e, -)) \simeq \mu(\mu(e,e), -) \simeq \mu(e, -)$ so is homotopy-idempotent, but it is also a weak equivalence, so is weakly homotopic to the identity.

\begin{proof}
The first part is an immediate consequence of Lemma \ref{lem:connecttivity-of-semisimplicial-amp}: the map $X_p \to *$ is $(k+2-p)$-connected for each $p$, and hence $\norm{X_\bullet} \to \norm{*_\bullet}$ is $(k+2)$-connected. But the geometric realisation of the terminal semi-simplicial space is contractible, and so $\norm{X_\bullet}$ is $(k+1)$-connected.

For the second part, we use the \emph{semi-simplicial path space} $PX_\bullet$. This is the semi-simplicial space $PX_p := X_{p+1}$, with face maps $d_i: PX_p \to PX_{p-1}$ given by those of $X_\bullet$ having the same names. The maps $d_{p+1}: PX_p \to X_p$ define a simplicial map $PX_\bullet \to X_\bullet$ and we will prove that it is homotopy-cartesian. To verify this, we use Lemma \ref{lem:minimal-checking-for-hocartesianness}. The condition for $(p,i)=(1,1)$ holds by hypothesis, so it remains to prove that the diagrams
\[
\xymatrix{
X_{p+1} \ar[r]^{d_0} \ar[d]^{d_{p+1}} & X_p \ar[d]^{d_p}\\
X_p \ar[r]^{d_0} & X_{p-1}
}
\]
are homotopy cartesian. Under the weak equivalences $\kappa_i$ (for $p-1 \leq i \leq p+1$), this diagram becomes 
\[
\xymatrix{
X_1^{p+1} \ar[rr]^{pr_{\{2, \ldots, p+1 \}}} \ar[d]^{pr_{\{1, \ldots, p \}}} & & X_1^{p} \ar[d]^{pr_{\{1, \ldots, p-1 \}}}\\
X_1^{p} \ar[rr]^{pr_{\{2, \ldots, p \}}} & & X_1^{p-1},
}
\] 
which is obviously homotopy cartesian. Therefore
\[
\xymatrix{
X_1 \ar[r] \ar[d]^{d_1} & \norm{PX_\bullet} \ar[d] \\
X_0 \ar[r] & \norm{X_\bullet}
}
\]
is homotopy cartesian, by Theorem \ref{lem:simplicial-hocartesianness}. 

We will now show that $\norm{PX_\bullet}$ is weakly contractible. Using the simplicial identities, one quickly checks that the maps $\eps_p = d_0^{p+1} : PX_p =X_{p+1} \to X_0$ form an augmentation $PX_\bullet \to PX_{-1} := X_0$. We shall show this is a weak equivalence by showing that $H_*(PX_{-1}, \norm{PX_\bullet};\bZ)=0$ and then showing that $\norm{PX_\bullet}$ is simply-connected: the claim then follows from Whitehead's Theorem.

To see that the homology of the pair $(PX_{-1}, \norm{PX_\bullet})$ vanishes, consider the morphism
$$g_{q+1} : X_q \overset{\kappa_q}\lra X_1^q \overset{e \times \mathrm{Id}}\lra X_1^{q+1} \overset{\kappa_{q+1}}{\underset{\simeq}{\longleftarrow}} X_{q+1}$$
in the homotopy category, where $e \in X_1$ represents the identity element of $\pi_0(X_1)$ as discussed above. This satisfies the identities of Lemma \ref{lem:ExtraDeg} \emph{up to weak homotopy}. Thus in the spectral sequence
$$E^1_{p,q} = H_p(PX_q;\bZ) \Longrightarrow H_{p+q}(PX_{-1}, \norm{PX_\bullet};\bZ)$$
the maps $g_{q+1}$ give a chain contraction of $(E^1_{p, *}, d^1)$, as we have
\begin{align*}
(g_q)_* d^1 + d^1 (g_{q+1})_* &= \left(\sum_{i=0}^q (-1)^i (g_q)_*(d_i)_*\right) + \left(\sum_{j=0}^{q+1} (-1)^j (d_j)_* (g_{q+1})_*\right)\\
 &= (d_0)_* (g_{q+1})_* + \sum_{i=0}^q (-1)^i \left((g_q)_*(d_i)_* - (d_{i+1})_* (g_{q+1})_*\right)\\
 &= \mathrm{Id}
\end{align*}
Thus $E^2_{*,*}=0$ and hence $H_{*}(PX_{-1}, \norm{PX_\bullet};\bZ)=0$ as claimed.

To show that $\norm{PX_\bullet}$ is simply-connected, let $PX'_\bullet$ be obtained by collapsing down the $0$-simplices of $PX_\bullet$ to a point. Consider the map of homotopy cofibre sequences
\begin{equation*}
\xymatrix{
X_1 \ar[d] \ar[r]& \norm{PX_\bullet} \ar[d] \ar[r]& \norm{PX_\bullet'} \ar[d]\\
\pi_0X_1 \ar[r] & \norm{\pi_0PX_\bullet} \ar[r]& \norm{\pi_0 PX_\bullet'}
}
\end{equation*}
The map $X_1 \to \pi_0X_1$ is 1-connected. The map $PX'_p \to \pi_0(PX'_p)$ is $(2-p)$-connected, so $\norm{PX'_\bullet} \to \norm{\pi_0 PX'_\bullet}$ is 2-connected by Lemma \ref{lem:connecttivity-of-semisimplicial-amp}. The semi-simplicial set $\pi_0PX_\bullet$ is in bijection with $\pi_0(X_1)^{p+1}$ in degree $p$, and can be identified with $E_\bullet \pi_0(X_1)$ for the group $\pi_0(X_1)$, so $\norm{\pi_0PX_\bullet} \simeq *$. Now the map $X_1 \to \norm{PX_\bullet}$ is nullhomotopic (it is homotopic to $x \mapsto \mu(x,e) \in X_1 \subset \norm{PX_\bullet}$ which in turn is homotopic to $x \mapsto \pi_2(x,e) = e \in X_1 \subset \norm{PX_\bullet}$), so the middle map is a retract of the right-hand map, so is also an isomorphism on fundamental groups. Thus $\norm{PX_\bullet}$ is simply-connected.
\end{proof}

\begin{remark}\label{rem:tothm:segal-deltaspace}
If $X_\bullet$ is a simplicial space satisfying (i) and (ii) of Theorem \ref{thm:SegalDeltaSpace}, then instead of the hypothesis of (iv) it is enough to just ask for the square
\[
\xymatrix{
X_2 \ar[r]^{d_1} \ar[d]^{d_2} & X_1 \ar[d]^{d_1}\\
X_1 \ar[r]^{d_1} & X_0
}
\]
to be homotopy cartesian, for the same conclusion to hold. This is because the maps $h_{p+1}=s_{p+1}: PX_{p}=X_{p+1} \to PX_{p+1}=X_{p+2}$ form a system of extra degeneracies, so Lemma \ref{lem:ExtraDeg} shows that the augmentation map $\norm{PX_\bullet} \to X_0$ is a weak homotopy equivalence, and we have assumed that $X_0 \simeq *$. 
\end{remark}

\section{(Non-unital) topological categories}\label{sec:TopCat}

\begin{defn}
A \emph{non-unital topological category} $\cC$ consists of an object space $\cC_0=\Ob (\cC)$, a morphism space $\cC_1 =\Mor(\cC)$ and three maps
\[
s,t: \cC_1 \lra \cC_0 \quad \text{ and } \quad m: \cC_1 \times_{\cC_0} \cC_1 := \{ (f,g) \in \cC_1 \times \cC_1 \,\vert\, t(f)=s(g)\ \lra \cC_1,
\]
such that
\[
m (m(f,g),h) = m(f,m(g,h));\quad t (m(f,g))= t (g); \quad s (m(f,g)) = s (f)
\]
for all $f,g,h$ for which these expressions are defined.
\end{defn}

One thinks of $s$ as the map associating to a morphism its source, $t$ as the map associating to a morphism its target, and $m$ as the composition of morphisms, whence we write $g \circ f := m(f,g)$. We write $\cC(b_0,b_1) := (s,t)^{-1} (b_0,b_1)$ for the space of morphisms from $b_0$ to $b_1$.
A functor $F: \cC \to \cD$ between non-unital topological categories is a pair of continuous maps $F_i: \cC_i \to \cD_i$, $i=0,1$ such that $s F_1 = F_0 s$, $t F_1 = F_0 t$, and $m\circ (F_1 \times F_1) = F_1 \circ m$. The set $\Fun(\cC, \cD)$ of functors is endowed with a topology as a subspace of $\mathrm{map}(\cC_0, \cD_0) \times \mathrm{map}(\cC_1, \cD_1)$.

\begin{defn}
A \emph{unital topological category} is a non-unital topological category $\cC$ together with a map $u: \Ob (\cC) \to \Mor (C)$ such that $t \circ u = s \circ u = \id$ and $m (f,u(t(f)))=f$ and $m (u(s (f)), f)=f$ for all $f \in \Mor (C)$. 
\end{defn}

We shall say, slightly informally, that $\cC$ has units if there is the structure of a unital topological category on it.

\begin{defn}
Let $\cC$ be a non-unital topological category. The \emph{(semi-simplicial) nerve} $N_\bullet \cC=\cC_{\bullet}$ of $\cC$ is the semi-simplicial space whose space of $p$-simplices is the space $\Fun ([p],\cC)$. For a morphism $\alpha: [q] \to [p]$, the map $\alpha^*: N_p \cC \to N_q \cC$ is given by precomposition with $\alpha$. 

The \emph{classifying space} $B\cC$ of $\cC$ is by definition the geometric realisation of its nerve, $B \cC := \norm{\cC_{\bullet}}$. A functor $F : \cC \to \cD$ induces a semi-simplicial map $F_\bullet: \cC_\bullet \to \cD_\bullet$ of semi-simplicial spaces and hence a map $BF : B\cC \to B\cD$ of classifying spaces.
\end{defn}

More explicitly, $N_0 \cC= \cC_0$, $N_1 \cC = \cC_1$, $d_1 =s, \, d_0 = t : \cC_1 \to \cC_0$. For higher values of $p$, $N_p \cC$ is the space $\cC_p := \cC_1 \times_{\cC_0} \cC_1 \times_{\cC_0} \cdots \times_{\cC_0} \cC_1$ ($p$ factors) with face maps given by composition, and even more explicitly, the points of $N_p \cC$ are the sequences 
\[
 c_0 \stackrel{f_1}{\lra} c_1 \lra \cdots \lra c_{p-1} \stackrel{f_p}{\lra} c_p
\]
of composable morphisms in $\cC$, and the face maps are given by
\[
 d_i (f_1, \ldots,f_p):=
 \begin{cases}
  (f_2, \ldots, f_p) & i=0\\
  (f_1, \ldots, f_{i+1} \circ f_{i}, f_p) & 0 < i <p\\
  (f_1, \ldots, f_{p-1}) & i=p.
 \end{cases}
\]
From this point of view, the data of a non-unital topological category is captured precisely by spaces of 0-, 1-, and 2-simplices of $\cC_\bullet$ and the face maps between them: the source and target maps are given by $d_1: \cC_1 \to \cC_0$ and $d_0: \cC_1 \to \cC_0$ respectively, and composition of morphisms is given by $d_1 : \cC_2 \to \cC_1$. For this reason we shall freely confuse the target and source maps with $d_0, d_1 : \cC_1 \to \cC_0$.

\begin{lem}\label{lem:NatTrans}
If $\eta : F \Rightarrow G : \cC \to \cD$ is a natural transformation, then there is an induced homotopy $BF \simeq BG : B\cC \to B\cD$ of maps on classifying spaces.
\end{lem}
\begin{proof}
We apply Lemma \ref{lem:semisimplicialhomotopy} with
$$h_{p+1, i}(f_1, f_2, \ldots, f_p) = (F(f_1), F(f_2), \ldots, F(f_i), \eta_{c_i}, G(f_{i+1}), \ldots,  G(f_p))$$
where the hypotheses are immediately verified.
\end{proof}

\subsection{Fibrancy conditions}

We shall only be able to make homotopical statements about the classifying spaces of (non-unital) topological categories when some of the structure maps involved are fibrations (see Remark \ref{rem:WhatIsAFibration} for a discussion of what ``fibration" can be taken to mean).

\begin{defn}
A non-unital topological category $\cC$ is called \emph{left fibrant} if the source map $d_1: \cC_1 \to \cC_0$ is a fibration. It is called \emph{right fibrant} if the target map $d_0: \cC_1 \to \cC_0$ is a fibration. 

Moreover, $\cC$ is called \emph{fibrant} if $(d_0,d_1): \cC_1 \to \cC_0 \times \cC_0$ is a fibration. 
\end{defn}

If $\cC$ is fibrant then it is both left and right fibrant, but the converse need not hold: consider the topological category with objects and morphisms a space $X$, and all structure maps the identity; this is always left and right fibrant, but is fibrant only if there are no non-constant paths in $X$.

\begin{lem}\label{inheritance-of-fibrances}
If $\cC$ is left fibrant, then $d_p: \cC_p \to \cC_{p-1}$ is a fibration. If $\cC$ is right fibrant, then $d_0: \cC_p \to \cC_{p-1}$ is a fibration. 
\end{lem}

\begin{proof}
The follows because the squares
\[
\xymatrix{
\cC_{p} \ar[d]^{d_p} \ar[r]^{d_0 \cdots d_0} & \cC_1 \ar[d]^{d_1} & \cC_{p} \ar[d]^{d_0} \ar[r]^{d_2 \cdots d_p} & \cC_1 \ar[d]^{d_0}\\
\cC_{p-1} \ar[r]^{d_0 \cdots d_0} & \cC_0 & \cC_{p-1} \ar[r]^{d_1 \cdots d_{p-1}} & \cC_0
}
\]
are cartesian.
\end{proof}

\subsection{The unitalisation}

If $\cC$ has units, then the semi-simplicial space $N_\bullet \cC$ has the structure of a simplicial space \cite{Segal-classifying}. Just as we can freely add degeneracies to a semi-simplicial object to form a simplicial one, we can freely add units to a non-unital topological category to form a unital one.

\begin{defn}
The \emph{unitalisation} of a non-unital topological category $\cC$ is the topological category $\cC^+$ with object space $\Ob (\cC^+)= \Ob (\cC)$ and morphism space $\Mor (\cC^+)= \Mor (\cC) \sqcup \Ob (\cC)$. The source and target maps are extended by the identity on $\Ob (\cC)$. The composition map $m^+ $ for $\cC^+$ is defined so that $c \in \Ob (\cC) \subset \Mor (\cC^+)$ behaves as the identity morphism at $c$.
\end{defn}

The category $\cC^+$ is never fibrant unless the object space $\Ob (\cC)$ has no non-constant paths. However, $\cC^+$ is left (or right) fibrant if $\cC$ is left (or right) fibrant. This limits the use of the unitalisation. But unitalisation has one very pleasant property, which we learnt from M.\ Krannich \cite[Lemma 1.3.11]{Krannich}.

\begin{proposition}\label{prop:krannich-lemma}
Let $\cC$ be a non-unital topological category. Then the natural map $B \cC \to B  \cC^+$ is a weak homotopy equivalence.
\end{proposition}
\begin{proof}
There is an isomorphism $N_\bullet \cC^+ \cong E(N_{\bullet} \cC)$ of simplicial spaces, such that the inclusion $N_\bullet \cC \to N_\bullet \cC^+$ corresponds to the unit map $N_\bullet \cC \to E N_\bullet \cC$. Apply Lemma \ref{lem:AdjUnitsTop}.
\end{proof}

\subsection{Soft units}\label{subsec:softunits}

From the point of view of the homotopy theory of classifying spaces of (unital, discrete) categories, such as Quillen's Theorems A and B, an important role is played by over-categories $\cC/c$ (and dually under-categories $c \backslash \cC$). 

Recall that for an object $c \in \Ob(\cC)$, the \emph{over-category} $\cC / c$ has objects the arrows $f:b \to c$, and morphisms $(g:a \to c) \to (f:b \to c)$ given by a morphism $h:a \to b$ such that $f \circ h = g$. This definition can be made equally well for non-unital topological categories, by topologising both objects and morphisms as subspaces of $\cC_1$. Dually (by reversing arrows), one defines the under-category $c \backslash \cC$.

If $\cC$ is a \emph{unital} topological category then $\cC/c$ has an object $\id_c : c \to c$ which is terminal: there is a natural transformation from $\id_{\cC/c}$ to the constant functor to $\id_c$. By Lemma \ref{lem:NatTrans} this gives a contraction of $B(\cC/c)$. Similarly, $B (c \backslash \cC)$ is contractible if $\cC$ is unital.

If $\cC$ is a \emph{non-unital} topological category then $B(\cC/c)$ need not be contractible: for example, it can be empty. Instead, we axiomatise this property as follows.

\begin{defn}\label{defn:soft-units}
A non-unital topological category $\cC$ has \emph{soft left units} if for each $c \in \cC_0$ we have $B(\cC/c) \simeq *$. It has \emph{soft right units} if for each $c \in \cC_0$ we have $B(c \backslash \cC) \simeq *$.
\end{defn}

We will describe a convenient property, more general than having units, which implies that a non-unital topological category has soft left or right units. This property arises naturally for non-unital topological categories such as cobordism categories.

\begin{lem}\label{lem:overcategories-contractible}
Let $\cC$ be a non-unital topological category and let $f \in \cC(c,c')$ be a morphism in $\cC$. Then the induced functor $f_*: \cC/c \to \cC/c'$ given by postcomposition with $f$ induces a nullhomotopic map on classifying spaces.

Therefore if an object $c \in \cC_0$ is either the source or target of a morphism $f$ which induces a weak equivalence on over-categories, it follows that $B(\cC/c) \simeq *$. The analogous statement holds for under-categories.
\end{lem}

\begin{proof}
We consider the case of the over-categories. There are maps $h_p: N_p (\cC/c) \to N_{p+1}(\cC/c')$, given by sending a $p$-tuple of composable morphisms $c_0 \to c_1 \to \cdots \to c_p \to c$ in $\cC/c$ to the $(p+1)$-tuple of composable morphisms $c_0 \to c_1 \to \cdots \to c_p \to c \overset{f}\to c'$ in $\cC/c'$. These form a semi-simplicial nullhomotopy from $N_\bullet f_*$ to the constant map to $(c \stackrel{f}{\to} c')$. Then apply Lemma \ref{lem:semisimplicialcontraction}.
\end{proof}

This observation may be applied to many non-unital topological categories arising in practice, because while they do not have units they do have many morphisms composition with which which induce weak equivalences on morphism spaces, as follows.

\begin{defn}
Let $\cC$ be a topological category. We say that $\cC$ has \emph{weak left units} if for each object $b \in \cC_0$, there is a morphism $u: b\to b'$ in $\cC$ so that the map
\[
\cC (-,b) := d_0^{-1} (b)\stackrel{u \circ -}{\lra}\cC (-,b')
\]
is a weak homotopy equivalence. Dually, $\cC$ has \emph{weak right units} if for each object $b \in \cC_0$, there is a morphism $u:b'\to b$ in $\cC$ such that
\[
\cC (b,-):=d_1^{-1}(b) \stackrel{- \circ u}{\lra} \cC (b',-)
\]
is a weak homotopy equivalence. 
\end{defn}

\begin{remark}\label{rem:weak-units-fibrant}
If $\cC$ is left fibrant, then $u \circ - : \cC(-,b) \to \cC(-,b')$ is a weak equivalence if and only if $u \circ - : \cC(a,b) \to \cC(a,b')$ is a weak equivalence for each $a \in \cC_0$.
\end{remark}

\begin{lem}\label{overcategories-nonunital}
If $\cC$ has weak left units and is right fibrant, then it has soft left units. 
Dually, if $\cC$ has weak right units and is left fibrant, then it has soft right units. 
\end{lem}

\begin{proof}
We only treat the first case. Let $u \in \cC(c,c')$ be a weak left unit. The squares
\[
\xymatrix{
N_p (\cC /c) \ar[r]^{N_p(u_*)}\ar[d]^{d_0^p} & N_p(\cC /c')\ar[d]^{d_0^p} \ar[r] & N_p(\cC) \ar[d]^-{d_0^p} \\
N_0 (\cC /c) \ar[r]^{N_0(u_*)}  & N_0 (\cC/c') \ar[r]& N_0(\cC)
}
\]
are both cartesian. By Lemma \ref{inheritance-of-fibrances} the right-hand vertical map is a fibration, and so all the vertical maps are fibrations and hence both squares are homotopy cartesian. We now consider the left-hand square: since the bottom horizontal map is a weak equivalence by assumption, it follows that the upper horizontal one is as well. Therefore, the functor $u_* : \cC/c \to \cC/c'$ induces a levelwise equivalence on nerves. But the map $Bu_*: B(\cC/c) \to B(\cC/c')$ is also nullhomotopic by Lemma \ref{lem:overcategories-contractible}, so $B(\cC/c) \simeq *$.
\end{proof}

\section{Quillen's Theorems A and B and bi-semi-simplicial resolutions}\label{sec:Resolution}

Let $F: \cC\to \cD$ be a functor of discrete and unital categories. Quillen's Theorem A \cite{Quillen} is a classical and well-known criterion to show that $BF: B \cC \to B \cD$ is a weak equivalence. Similarly, Quillen's Theorem B \cite{Quillen} is a device to identify the homotopy fibre of $BF$. In this section, we prove generalisations of Quillen's Theorems for topological and nonunital categories. Those are stated as Theorems \ref{thm:quillena}, \ref{thm:quillena1} and \ref{thm:quillenb} below, but before we can state them precisely, we need to introduce a construction that is used in the proofs. 

\begin{defn}
Let $F: \cC \to \cD$ be a continuous functor between non-unital topological categories. Let $(F/\cD)_{p,q}$ be the space of all pairs in $N_p\cC \times N_{q+1}\cD$ of the form $(a_0 \to \cdots \to a_p, F(a_p) \to b_0 \to \cdots \to b_q)$ (of course, the unnamed arrows are part of the data). The $(F/\cD)_{p,q}$ form, in an evident way, a bi-semi-simplicial space. It has augmentation maps 
\[
\eps_{p,q}: (F/\cD)_{p,q} \lra \cC_p; \; (a_0 \to \cdots \to a_p, F(a_p) \to b_0 \to \cdots \to b_q) \longmapsto (a_0 \to \cdots \to a_p)
\]
and 
\[
\eta_{p,q}: (F/\cD)_{p,q} \lra \cD_q; \; (a_0 \to \cdots \to a_p, F(a_p) \to b_0 \to \cdots \to b_q) \longmapsto (b_0 \to \cdots \to b_q).
\]

Dually, let $(\cD/F)_{p,q}$ be the space of all pairs in $N_p\cC \times N_{q+1}\cD$ of the form $(a_0 \to \cdots \to a_p,  b_0 \to \cdots \to b_q \to F(a_0))$. The $(\cD/F)_{p,q}$ form, in an obvious way, a bi-semi-simplicial space. It has augmentation maps 
\[
\xi_{p,q}: (\cD/F)_{p,q} \lra \cC_p; \; (a_0 \to \cdots \to a_p,  b_0 \to \cdots \to b_q \to F(a_0)) \longmapsto (a_0 \to \cdots \to a_p)
\]
and 
\[
\zeta_{p,q}: (\cD/F)_{p,q} \lra \cD_q; \;(a_0 \to \cdots \to a_p,  b_0 \to \cdots \to b_q \to F(a_0))  \longmapsto (b_0 \to \cdots \to b_q).
\]
\end{defn}

For the rest of this section we shall makes statements about both constructions, but only prove them in the first case: the second is dual.

\begin{lem}\label{lem:augmentation-trick-triangle-commutes}
The diagrams
\[
\xymatrix{
 & \norm{(F/\cD)_{\bullet,\bullet}} \ar[dl]_{\norm{\eps_{\bullet,\bullet}}} \ar[dr]^{\norm{\eta_{\bullet,\bullet}}} & \\
\norm{\cC_{\bullet}} \ar[rr]^{\norm{ F_\bullet }} & & \norm{\cD_{\bullet}}
}
\]
and 
\[
\xymatrix{
 & \norm{(\cD/F)_{\bullet,\bullet}} \ar[dl]_{\norm{\xi_{\bullet,\bullet}}} \ar[dr]^{\norm{\zeta_{\bullet,\bullet}}} & \\
\norm{\cC_{\bullet}} \ar[rr]^{\norm{F_\bullet }} & & \norm{\cD_{\bullet}}
}
\]
are (naturally) homotopy commutative.
\end{lem}

\begin{proof}
For $p,q\geq 0$, we define a map
\[
H_{p,q}: I \times (F/\cD)_{p,q} \times \Delta^p \times \Delta^q \lra \norm{\cD_\bullet}
\]
by sending $(t;a_0\to\cdots\to a_p,F(a_p)\to b_0 \to \cdots \to b_q;r,s)$ to
\[
(F(a_0) \to \cdots \to F(a_p)\to b_0 \to \cdots \to b_q;tr,(1-t)s) \in  \cD_{p+q+1} \times \Delta^{p+q+1}.
\]
This respects the simplicial relations and hence descends to a map $H: I \times \norm{(F/\cD)_{\bullet,\bullet}} \to \norm{\cD_\bullet}$ (we have used that taking products preserves quotient maps in the category of compactly generated spaces). This satisfies $H(0,-)=\norm{\eta_{\bullet,\bullet}}$ and $H(1,-)= \norm{F_\bullet} \circ \norm{\eps_{\bullet,\bullet}}$.
\end{proof}

\begin{lem}\label{lem:map-from-resolution-back-to-source-is-equiv}
If $\cD$ is unital, then $\norm{\eps_{\bullet,\bullet}} : \norm{(F/\cD)_{\bullet,\bullet}} \to \norm{\cC_{\bullet}}$ and $\norm{\xi_{\bullet,\bullet}} : \norm{(\cD/F)_{\bullet,\bullet}} \to \norm{\cC_{\bullet}}$ are weak homotopy equivalences. 
\end{lem}
\begin{proof}
By Theorem \ref{thm:levelwiseequivalence}, it is enough to prove that $\norm{(F/\cD)_{p,\bullet}}\to \cC_p$ is a weak homotopy equivalence for all $p$. We we show that the augmented semi-simplicial space $\epsilon_{p, \bullet} : (F/\cD)_{p,\bullet}\to \cC_p$ has an extra degeneracy of the second type described in Lemma \ref{lem:ExtraDeg}. Define $g_0 : \cC_p \to (F/\cD)_{p,0}$ by
$$(a_0\to \cdots\to a_p) \longmapsto (a_0\to \cdots\to a_p, F(a_p)\stackrel{\id}{\to}F(a_p))$$
and $g_{q+1} : (F/\cD)_{p,q} \to (F/\cD)_{p,q+1}$ by 
\begin{align*}
&(a_0\to \cdots \to a_p, F(a_p)\to b_0 \to \cdots \to a_q)\\
&\quad \longmapsto (a_0\to \cdots \to a_p,F(a_p)\stackrel{\id}{\to} F(a_p)\to b_0 \to \cdots \to a_q).
\end{align*}
These satisfy the conditions in Lemma \ref{lem:ExtraDeg}, showing that $\norm{(F/\cD)_{p,\bullet}}\to \cC_p$ is a homotopy equivalence.
\end{proof}

For non-unital categories, the conclusion of Lemma \ref{lem:map-from-resolution-back-to-source-is-equiv} does not hold without further hypotheses. If we do not have units then, rather than the explicit homotopy coming from an extra degeneracy used in the proof of the last lemma, note that for $a=(a_0 \to \cdots \to a_p) \in \cC_p$, we have
\[
\eps_{p,\bullet}^{-1}(a)= N_\bullet (F(a_p) / \cD)
\]
and 
\[
\xi_{p,\bullet}^{-1}(a)= N_\bullet (\cD / F(a_0)),
\]
the semi-simplicial nerves of over- and under-categories. We have axiomatised the contractibility of these as soft left- or right-units, and we will show that under appropriate  fibrancy conditions this is enough to get the conclusion of Lemma \ref{lem:map-from-resolution-back-to-source-is-equiv}.

\begin{lem}\label{lem:map-from-resolution-to-source-fibraton}
If $\cD$ is left fibrant, then the augmentation map $\eps_{p,q}: (F/\cD)_{p,q}\to \cC_p$ is a fibration. 
If $\cD$ is right fibrant, then the augmentation map $\xi_{p,q}: (\cD/F)_{p,q}\to \cC_p$ is a fibration. 
\end{lem}

\begin{proof}
Observe that both squares
\[
\xymatrix{
(F/\cD)_{p,q} \ar[d]^{\eps_{p,q}} \ar[r]^{d_0^p} & (F/\cD)_{0,q} \ar[r]^{\gamma}\ar[d]^{\eps_{0,q}} &\cD_{q+1} \ar[d]^{d_1 \cdots d_{q+1}}\\
\cC_p \ar[r]^{d_0^p} & \cC_0 \ar[r]^{F_0} & \cD_0,
}
\]
where $\gamma ( a_0, F(a_0) \to b_0 \to \cdots \to b_q) :=(F(a_0) \to b_0 \to \cdots \to b_q)$, are cartesian, and use Lemma \ref{inheritance-of-fibrances}.
\end{proof}

\begin{cor}\label{lem:map-from-resolution-back-to-source-is-equiv-nonunital}
If $\cD$ is left fibrant and has soft right units, then $\norm{\eps_{\bullet,\bullet}}: \norm{(F/\cD)_{\bullet,\bullet}} \to B\cC$ is a weak equivalence. Dually, if $\cD$ is right fibrant and has soft left units, then $\norm{\xi_{\bullet,\bullet}}: \norm{(\cD/F)_{\bullet,\bullet}} \to B\cC$ is a weak equivalence.\qedhere
\end{cor}
\begin{proof}
By Lemma \ref{lem:map-from-resolution-to-source-fibraton} the maps $\eps_{p,q}: (F/\cD)_{p,q}\to \cC_p$ are fibrations, so Lemma \ref{lem:augmented-fibration} applies to $\eps_{p,\bullet}: (F/\cD)_{p,\bullet}\to \cC_p$ and so for each $a=(a_0 \to \cdots \to a_p) \in \cC_p$ the map
$$B(F(a_p)/\cD) = \norm{\eps_{p,\bullet}^{-1}(a)} \lra \hofib_a \norm{\eps_{p,\bullet}}$$
is a weak equivalence. But as $\cD$ has soft right units the source of this map is contractible, and hence $\norm{(F/\cD)_{p,\bullet}} \to \cC_p$ is a weak equivalence. Then claim then follows by geometrically realising in the $p$-direction and using Theorem \ref{thm:levelwiseequivalence}.
\end{proof}

To make use of these resolutions, we shall also need to know that the maps $\eta_{p,q}$ and $\zeta_{p,q}$ are fibrations, and the final result of this section is a criterion for this to hold.

\begin{lem}\label{lem:map-to-target-fibration}
\mbox{}
\begin{enumerate}[(i)]
\item If $\eta_{p,0}$ is a fibration, then so is $\eta_{p,q}$ for all $q \geq 0$.

\item If $\eta_{0,0}$ is a fibration and $\cC$ is right fibrant, then $\eta_{p,0}$ is a fibration, for all $p \geq 0$.

\item If $F_0: \cC_0 \to \cD_0$ is a fibration and $\cD$ is right fibrant, then $\eta_{0,0}$ is a fibration.
\end{enumerate}
Dually,
\begin{enumerate}[(i)]
\setcounter{enumi}{3}
\item If $\zeta_{p,0}$ is a fibration, then so is $\zeta_{p,q}$ for all $q \geq 0$.

\item If $\zeta_{0,0}$ is a fibration and $\cC$ is left fibrant, then $\zeta_{p,0}$ is a fibration, for all $p \geq 0$.

\item If $F_0: \cC_0 \to \cD_0$ is a fibration and $\cD$ is left fibrant, then $\zeta_{0,0}$ is a fibration.
\end{enumerate}
\end{lem}

\begin{proof}
The square
\[
\xymatrix{
(F/\cD)_{p,q} \ar[rr]^{d_1 \cdots d_q} \ar[d]^{\eta_{p,q}} & & (F/\cD)_{p,0} \ar[d]^{\eta_{p,0}}\\
\cD_q \ar[rr]^{d_1 \cdots d_q} & & \cD_0
}
\]
is cartesian, which proves (i). For part (ii), use that
\[
\xymatrix{
(F/\cD)_{p,0} \ar[r]^{\eps_{p,0}} \ar[d]^{d_0 \cdots d_0} & \cC_p \ar[d]^{d_0 \cdots d_0}\\
(F/\cD)_{0,0} \ar[r]^{\eps_{0,0}} & \cC_0
}
\]
is cartesian, Lemma \ref{inheritance-of-fibrances}, and that $\eta_{p,0} = \eta_{0,0} \circ (d_0)^p$. For part (iii), let $\gamma: (F/\cD)_{0,0}\to \cD_1$ be given by $\gamma (a,F(a)\to b)= (F(a)\to b)$. The diagram
\[
\xymatrix{
(F/\cD)_{0,0} \ar[r]^{\gamma} \ar[d]^{\eps_{0,0}} &  \cD_1 \ar[d]^{d_1}\\
\cC_0 \ar[r]^{F_0} &\cD_0
}
\]
is cartesian, so $\gamma$ is a fibration, hence so is $d_0 \circ \gamma = \eta_{0,0}$. 
\end{proof}

We can now state and prove our version of Quillen's Theorems A and B for non-unital topological categories.

\begin{thm}[Quillen's Theorem A]\label{thm:quillena}
Let $F: \cC \to \cD$ be a continuous functor. Assume that
\begin{enumerate}[(i)]
 \item $B(F / b)$ is contractible for each $b \in \cD_0$,
 \item $\norm{\eps_{\bullet,\bullet}}: \norm{(F/\cD)_{\bullet,\bullet}} \to B \cC$ is a weak equivalence,  
 \item $\eta_{p,0}: (F/\cD)_{p,q} \to N_q \cD$ is a fibration for each $p \geq 0$.
\end{enumerate}
Then $BF: B \cC \to B \cD$ is a weak homotopy equivalence.

Conditions (ii) and (iii) are satisfied if either
\begin{enumerate}[(i)]
\setcounter{enumi}{3}
 \item $\cC$ is right fibrant, $\cD$ is left fibrant and has soft right units and $\eta_{0,0}$ is a fibration or
 \item $\cC $ is right fibrant, $\cD$ has units and $\eta_{0,0}$ is a fibration.
\end{enumerate}
\end{thm}

There is a dual version, with a parallel proof.

\begin{thm}[Quillen's Theorem A, dual version]\label{thm:quillena1}
Let $F: \cC \to \cD$ be a continuous functor. Assume that
\begin{enumerate}[(i)]
 \item $B(b/F)$ is contractible for each $b \in \cD_0$,
 \item $\norm{\xi_{\bullet,\bullet}}: \norm{(\cD/F)_{\bullet,\bullet}} \to B \cC$ is a weak equivalence,  
 \item $\zeta_{p,q}: (\cD/F)_{p,q} \to N_q \cD$ is a fibration for each $p,q \geq 0$.
\end{enumerate}
Then $BF: B \cC \to B \cD$ is a weak homotopy equivalence.

Conditions (ii) and (iii) are satisfied if either
\begin{enumerate}[(i)]
\setcounter{enumi}{3}
 \item $\cC$ is left fibrant, $\cD$ is right fibrant and has soft left units and $\zeta_{0,0}$ is a fibration or
 \item $\cC $ is left fibrant, $\cD$ has units and $\zeta_{0,0}$ is a fibration.
\end{enumerate}
\end{thm}

In the case of discrete (unital) categories, this is a classical result of Quillen \cite{Quillen}. A version for (unital) simplicial categories was proven by Waldhausen \cite[\S 4]{Waldhausen}.

\begin{proof} [Proof of Theorem \ref{thm:quillena}]
That conditions (iv) or (v) imply conditions (ii) and (iii) follows from Lemmas \ref{lem:map-from-resolution-back-to-source-is-equiv}, Lemma \ref{lem:map-to-target-fibration} and Corollary \ref{lem:map-from-resolution-back-to-source-is-equiv-nonunital}. 

By Lemma \ref{lem:augmentation-trick-triangle-commutes}, it is enough to prove that $\norm{\eta_{\bullet,\bullet}} : \norm{(F/\cD)_{\bullet, \bullet}} \to \norm{\cD_\bullet}$ is a weak equivalence. Since each $\eta_{p,0}$ is a fibration, it follows by Lemma \ref{lem:map-to-target-fibration} that $\eta_{p,q}$ is a fibration for all $p,q \geq 0$, so by Lemma \ref{lem:augmented-fibration} for each $b= (b_0 \to \cdots \to b_q )\in \cD_q$ the natural map
\[B (F/ b) = \norm{\eta_{\bullet,q}^{-1}(b)} \lra \hofib_{b} \norm{\eta_{\bullet,q}}\]
is a weak equivalence. The source is contractible by assumption, so $\norm{\eta_{\bullet,q}}$ is a weak equivalence.
\end{proof}

Quillen's Theorem B \cite{Quillen} gives a criterion for identifying the homotopy fibre of a functor between ordinary categories. We now state and prove a version of this for non-unital topological categories; in fact we give a mild generalisation, due to Blumberg--Mandell \cite [Theorem 4.5]{BluMan}. In this case we only state one version: it has a dual version which we leave to the reader.

\begin{thm}[Quillen's Theorem B]\label{thm:quillenb}
Let 
\[
\xymatrix{
 \cA \ar[r]^{J} \ar[d]^{G} & \cC \ar[d]^{F}\\
 \cB \ar[r]^{H} & \cD
}
\]
be a commuting square of non-unital topological categories. Assume that 
\begin{enumerate}[(i)]
 \item $\cB$ and $\cD$ are left fibrant and have soft right units.
 \item $\cA$ and $\cC$ are right fibrant, and the maps $\eta_{0,0}: (G/\cB)_{0,0} \to \cB_0$ and $\eta_{0,0}: (F/\cD)_{0,0}\to \cD_0$ are fibrations.
 \item For each morphism $u: d \to d'$ in $\cD$, the functor $u_* : F/d \to F/d'$ induced by composition with $u$ induces a weak equivalence on classifying spaces.
 \item For each object $b \in \cB_0$, the functor $G/b \to F/H(b)$ induced by $J$ and $H$ is a weak equivalence.
\end{enumerate}
Then the square
\[
 \xymatrix{
B \cA \ar[r]^{BJ} \ar[d]^{BG} & B\cC \ar[d]^{BF}\\
B \cB \ar[r]^{BH} & B\cD
 }
\]
is homotopy cartesian.
\end{thm}

\begin{proof}
Using the resolutions of the functors $F$ and $G$, by assumption (i) and Corollary \ref{lem:map-from-resolution-back-to-source-is-equiv-nonunital} it is enough to show that the square
\[
 \xymatrix{
  \norm{(G/\cB)_{\bullet,\bullet}} \ar[r] \ar[d]^{\norm{\eta^G_{\bullet,\bullet}}}&  \norm{(F/\cD)_{\bullet,\bullet}}\ar[d]^{\norm{\eta^F_{\bullet,\bullet}}}\\
  B \cB \ar[r]^{BH} & B \cD
  }
\]
is homotopy cartesian. Arguing as in the proof of Theorem \ref{thm:quillena}, which requires assumption (ii), we see that the maps
\[
\norm{\eta^F_{\bullet,q}}: \norm{(F/\cD)_{\bullet,q}} \lra \cD_q \quad\quad\quad  \norm{\eta^G_{\bullet,q}}:\norm{(G/\cB)_{\bullet,q}} \lra \cB_q
\]
are quasifibrations. In the commutative square
\begin{equation}\label{diagram:proofquillenb}
\begin{gathered}
\xymatrix{
\norm{(F/\cD)_{\bullet,q}} \ar[r]^{d_i} \ar[d]^{\norm{\eta^F_{\bullet,q}}}  & \norm{(F/\cD)_{\bullet,q-1}} \ar[d]^{\norm{\eta^F_{\bullet,q-1}}} \\
\cD_q \ar[r]^{d_i} & \cD_{q-1}
}
\end{gathered}
\end{equation}
the fibre over $x=(d_0 \stackrel{u_1}{\to} \cdots \stackrel{u_q}{\to} d_q) \in \cD_q$ is $B(F/d_0)$, and the induced map on fibres is either the identity (if $i>0$) or it is the fibre transport map $(u_1)_* : B(F/d_0) \to B(F/d_1)$, which is a weak equivalence by assumption (iii). Therefore by Theorem \ref{lem:simplicial-hocartesianness}, the squares
\begin{equation}\label{diagram:proofquillenb2}
 \begin{gathered}
\xymatrix{
 B (F/d) \ar[d] \ar[r] & \norm{(F/\cD)_{\bullet,0}} \ar[d] \ar[r] & \norm{(F/\cD)_{\bullet,\bullet}} \ar[d]^{\norm{\eta^F_{\bullet,\bullet}}} \\
 \{d\} \ar[r] & \cD_0 \ar[r] & B\cD
 }
\end{gathered}
\end{equation}
are both homotopy cartesian. For each morphism $f:b \to b'$ in $\cB$, the induced map $f_* : B (G/b) \to B (G/b')$ is a weak equivalence, since it fits into a commutative diagram
\[
 \xymatrix{
 B (G/b) \ar[r]^-{\simeq} \ar[d]^-{f_*} & B(F /Hb) \ar[d]^-{H(f)_*}_{\simeq}\\
 B (G/b') \ar[r]^-{\simeq} & B(F/Hb')
 }
\]
in which all other maps are weak equivalences by assumption (iii) and (iv). Therefore, in the analogue of the diagram (\ref{diagram:proofquillenb2}) for the functor $G$ both squares are also homotopy cartesian. For $b \in \cB_0$ the composition
\[
 B(G/b) \stackrel{\simeq}{\lra} \hofib_b \norm{\eta^G_{\bullet,\bullet}} \stackrel{BJ}{\lra} \hofib_{Hb} \norm{\eta^F_{\bullet,\bullet}}
\]
is equal to the composition
\[
 B (G/b) \stackrel{\simeq}{\lra} B F/(Hb) \stackrel{\simeq}{\lra} \hofib_{Hb} \norm{\eta^F_{\bullet,\bullet}}.
\]
Therefore, $BJ : \hofib_b \norm{\eta^G_{\bullet,\bullet}} \to \hofib_{Hb} \norm{\eta^F_{\bullet,\bullet}}$ is a weak equivalence for each $b \in \cB_0$; since the inclusion $\iota : \cB_0 \to B \cB$ is $0$-connected, this finishes the proof.
\end{proof}

\begin{remark}\label{rem:WhatIsAFibration}
We wish to record a technical point about the meaning of the term ``fibration" in Theorems \ref{thm:quillena} and \ref{thm:quillenb} (which is also implicitly used in the term ``fibrant"). While we have in mind Serre fibrations, what is used in the argument is: Hurewicz fibrations are ``fibrations"; ``fibrations" are preserved under pullback; composition of ``fibrations" are ``fibrations"; ``fibrations" are quasifibrations. For example, this allows one to take the class of Dold fibrations or, even more generally, Dold--Serre fibrations (i.e.\ maps which have the weak covering homotopy property with respect to discs).
\end{remark}

\section{Base changing spaces of objects}

For a non-unital topological category $\cC$ and a continuous map $f : X \to \cC_0$, we may form a new non-unital topological category $\cC^X$ as follows. We let $\cC^X_0$ be $X$, and $F_0 : \cC^X_0  \to \cC_0$ be $f$. Then we define $\cC^X_1$ as the pullback
\begin{equation}\label{eq:PullbackCat}
\begin{gathered}
\xymatrix{
\cC^X_1  \ar[r]^-{F_1} \ar[d] & \cC_1 \ar[d]^-{s \times t}\\
\cC^X_0 \times \cC^X_0 \ar[r]^-{F_0 \times F_0} & \cC_0 \times \cC_0.
}
\end{gathered}
\end{equation}
The left-hand maps define $s, t : \cC^X_1 \to \cC^X_0$, and the universal property of the pullback provides a map $c : \cC^X_1 \times_{\cC^X_0} \cC^X_1\to \cC^X_1$; this defines a non-unital topological category, and the $F_i$ define a continuous functor $F : \cC^X \to \cC$. (If $\cC$ had units, then $\cC^X$ does too.) 

\begin{thm}
If $\cC$ is fibrant and has weak right (or left) units, and $f$ is 0-connected, then $BF : B\cC^X \to B\cC$ is a weak equivalence.
\end{thm}

\begin{proof}
We consider the resolution $(F/\cC)_{\bullet, \bullet}$ of the functor $F$. As $\cC$ is left fibrant and has weak right (say) units, it has soft right units by Lemma \ref{overcategories-nonunital}, and so Corollary \ref{lem:map-from-resolution-back-to-source-is-equiv-nonunital} applies and shows that $\norm{\epsilon_{\bullet, \bullet}} : \norm{(F/\cC)_{\bullet, \bullet}} \to B\cC^X$ is a weak equivalence. It remains to show that $\norm{\eta_{\bullet, \bullet}} : \norm{(F/\cC)_{\bullet, \bullet}} \to B\cC$ is a weak equivalence.

The space $(F/\cC)_{0,0}$ fits in to a cartesian square
\begin{equation*}
\xymatrix{
(F/\cC)_{0,0} \ar[r] \ar[d]^-{\epsilon_{0,0} \times \eta_{0,0}}& \cC_1 \ar[d]^-{s \times t}\\
\cC^X_0 \times \cC_0 \ar[r]^-{F_0 \times \mathrm{id}} & \cC_0 \times \cC_0,
}
\end{equation*}
and as $\cC$ is fibrant the right-hand vertical map is a fibration, and so $\eta_{0,0}$ is a fibration too. Furthermore, as $\cC$ is fibrant, \eqref{eq:PullbackCat} shows that $\cC^X$ is too. Hence by applying Lemma \ref{lem:map-to-target-fibration} (ii) then (i), each $\eta_{p,q}$ is a fibration. Hence, by Lemma \ref{lem:augmented-fibration}, for each $b = (b_0 \to \cdots \to b_q) \in \cC_q$ the map
$$B(\cC^X/b_0) = \norm{\eta_{\bullet,q}^{-1}(b)} \lra \hofib_b\norm{\eta_{\bullet,q}}$$
is a weak equivalence, so it is enough that the over-categories $B(\cC^X/b_0)$ be contractible for some object $b_0 \in \cC_0$ in each path-component. As $f : X \to \cC_1$ is 0-connected, we may suppose that $b_0 = F(x_0)$, but in this case $\cC^X/F(x_0) = \cC^X/x_0$, by \eqref{eq:PullbackCat}, so it is enough to show that $\cC^X$ has soft right units. As $\cC$ has weak right units so does $\cC^X$ (by Remark \ref{rem:weak-units-fibrant} and because both categories are fibrant), so by Lemma \ref{overcategories-nonunital} $\cC^X$ has soft right units as required.
\end{proof}

A typical application of this result is to take $X = \cC_0^\delta$ to be the set of objects of $\cC$ with the discrete topology, and $f : \cC_0^\delta \to \cC_0$ to be the identity function (which is 0-connected). This yields a category $\cC^\delta$ with discrete space of objects but the same space of maps between any two objects, which has a homotopy equivalent classifying spaces under the conditions given above.

\section{The Group-Completion Theorem}\label{sec:groupcompletion}

We shall take care to formulate and prove the group-completion theorem, and the main technical result underlying it, for homology with local coefficients. We therefore make the following definitions.

\begin{defn}
Let $\cL$ be a local coefficient system of $R$-modules on a space $X$. 
\begin{enumerate}[(i)]
\item The \emph{monodromy} of $\cL$ at $x \in X$ if the homomorphism $ \mu_x:\pi_1 (X,x) \to \Aut_{R\text{-Mod}} (\cL(x))$ induced from $\cL$. 
\item $\cL$ is called \emph{constant} if all monodromy homomorphisms are trivial.
\item $\cL$ is called \emph{abelian} if the images of all monodromy homomorphisms are abelian groups. 
\end{enumerate}
\end{defn}

\begin{assumption}
In the sequel, let $\cA$ be either 
\begin{enumerate}[(i)]
\item the class of constant local coefficient systems of $R$-modules, or
\item the class of abelian local coefficient systems of $R$-modules, or
\item the class of all local coefficient systems of $R$-modules.
\end{enumerate}
We say a map $f : X \to Y$ is an \emph{$\cA$-equivalence} if for every local coefficient system $\cL$ on $Y$ in the class $\cA$, the map
$$f_* : H_*(X ; f^*\cL) \lra H_*(Y;\cL)$$
is an isomorphism.
\end{assumption}

\begin{defn}
A commutative square of spaces
\[
 \xymatrix{
 W \ar[r] \ar[d]^{g} & X \ar[d]^{f} \\
 Z \ar[r]^{h} & Y
  }
\]
is called \emph{$\cA$-cartesian} if the induced map $\hofib_z (g) \to \hofib_{h(z)} (f)$ is an $\cA$-equivalence, for all $z \in Z$.
\end{defn}

\begin{remark}
Unlike for homotopy cartesian diagrams, the symmetry explained in Remark \ref{remark.hocartesianness-symmetric} does not generally hold for $\cA$-cartesian diagrams (though it does in case (iii)). A counterexample in case (i) is $R=\bZ$ when $W=Z=Y=*$ and $X=BG$ is the classifying space of an infinite acyclic group.
\end{remark}

The following homological analogue of Theorem \ref{lem:simplicial-hocartesianness} is the technical heart of the ``group-completion theorem'' and is due to McDuff and Segal \cite{McDuffSeg}. The notion of an $\cA$-cartesian map $f_\bullet: X_\bullet \to Y_\bullet$ is defined in analogy to Definition \ref{defn:hocartesian.morpism}.

\begin{thm}\label{thm:Acart}
If $f_\bullet: X_\bullet \to Y_\bullet$ is a $\cA$-cartesian map of semi-simplicial spaces, then the diagram
\begin{equation}\label{eq:good}
\begin{gathered}
 \xymatrix{
 X_0 \ar[r] \ar[d]^{f_0} & \norm{X_\bullet} \ar[d]^{\norm{f_\bullet}} \\
 Y_0 \ar[r] & \norm{Y_\bullet}
 }
\end{gathered}
\end{equation}
is $\cA$-cartesian. 
\end{thm}

The presentation of McDuff--Segal omits many details, to say the least. A more detailed exposition of the proof, with some imprecisions fixed, can be found in \cite{MilPal}. These proofs involve some fairly complicated point-set topology. There are proofs of an analogous result in the context of bi-simplicial sets by Jardine \cites{Jardine, Goerss-Jardine},  Moerdijk \cite{Moerdijk}, and Pitsch--Scherer \cite{PitschScherer}. These proofs use heavy machinery from simplicial homotopy theory (either model structures on the category of bi-simplicial sets, or (unpublished) results for manipulating homotopy colimits). The proof we shall give is essentially that of McDuff--Segal, but our argument replaces the point-set topology considerations with simplicial arguments. 

\subsection{Proof of Theorem \ref{thm:Acart}}

The main portion of the proof of Theorem \ref{thm:Acart} will be to prove the following version for simplicial spaces; the last step is the generalization to semi-simplicial spaces. We shall say that a map $f_\bullet:X_\bullet \to Y_\bullet$ of simplicial spaces is $\cA$-cartesian if the underlying map of semi-simplicial spaces has this property.

\begin{proposition}\label{prop:groupcompletion-simplicialanalog}
Let $f_\bullet: X_\bullet \to Y_\bullet$ be an $\cA$-cartesian map of simplicial spaces. Then the diagram \eqref{eq:good} is $\cA$-cartesian. 
\end{proposition}

The proof will be sequence of lemmas, each of which extends the class of base spaces $Y_\bullet$ for which the conclusion of Proposition \ref{prop:groupcompletion-simplicialanalog} holds. To this end, let us say that a simplicial space $Y_\bullet$ is \emph{basic} if for every  $\cA$-cartesian map of simplicial spaces $f_\bullet: X_\bullet \to Y_\bullet$ the diagram \eqref{eq:good} is $\cA$-cartesian. Given this definition, the statement of Proposition \ref{prop:groupcompletion-simplicialanalog} is that every simplicial space is basic.

\begin{lem}\label{groupcompletionproof-lemma1}
If $Y_\bullet$ is a simplicial set with contractible geometric realisation then it is basic. 
\end{lem}

\begin{proof}
The proof only uses the semi-simplicial structure. Let $y \in Y_0$ be a basepoint. Since $\norm{Y_\bullet}$ is contractible, the natural map $\eta:\hofib_y (\norm{f_\bullet}) \to \norm{X_\bullet}$ is a weak equivalence. Hence any coefficient system $\cL'$ on $\hofib_y (\norm{f_\bullet})$ is of the form $\eta^* \cL$ for a coefficient system on $\norm{X_\bullet}$, and if $\cL'$ lies in the class $\cA$, then so does $\cL$.
Therefore, we have to prove that for each point $y \in Y_0$, the inclusion map $j:f_0^{-1}(y) \to \norm{X_\bullet}$ induces an isomorphism $H_* (f_0^{-1}(y);j^*\cL) \to H_* (\norm{X_\bullet};\cL)$. 

The spectral sequence of the semi-simplicial space $X_\bullet$ with coefficients in $\cL$ discussed in Section \ref{subsec:spectralsequence} takes the form
\[
 E^1_{p,q} = H_q (X_p;\cL_p) \Rightarrow H_{p+q}(\norm{X_\bullet};\cL).
\]
Since $Y_p$ is discrete, we can write the $E^1$-term as
\[
  H_q (X_p;\cL_p) = \bigoplus_{s \in Y_p} H_q (f^{-1}(s);\cL_p|_{f^{-1}(s)}).
\]
To simplify notation, we write $H_q (f^{-1}(s);\cL_p):= H_q (f^{-1}(s);\cL_p|_{f^{-1}(s)})$. Because the map $f_\bullet$ is $\cA$-cartesian, the map $H_q (f^{-1}(s);\cL_p) \to H_q (f^{-1}(d_i s);\cL_{p-1})$ induced by the face map $d_i$ is an isomorphism. Hence $s \mapsto H_q (f^{-1} (s);\cL_p)$ is a locally constant coefficient system $H_q (f;\cL)$ on the simplicial set $Y_\bullet$. Hence $E^2_{p,q}= H_p (\norm{Y_\bullet}; H_q (f;\cL))$. Because $\norm{Y_\bullet}$ is contractible, it follows that $E^2_{p,q}=0$ for $p>0$. If $y \in Y_0$ is a basepoint, the induced map $\Delta^0_\bullet \to Y_\bullet$ of simplicial sets gives a comparison diagram
\[
 \xymatrix{
 f^{-1}(s) \ar[r] \ar[d] & X_\bullet \ar[d]^{f_\bullet}\\
 \Delta^0_\bullet \ar[r]^{y} & Y_\bullet.
 }
\]
It induces an isomorphism on the $E^2$-term of the spectral sequence, and therefore $f^{-1}(y) \to \norm{X_\bullet}$ induces an isomorphism in homology with coefficients in $\cL$, as claimed.
\end{proof}

The next step is a discretisation argument. For a simplicial space $Y_\bullet$, we consider the bi-simplicial set $(p,q)\mapsto \Sing_q Y_p$ and the associated diagonal simplicial set $\delta Y_p:=\Sing_p Y_p$. By Theorem \ref{thm:eilenberg-zilber}, Lemma \ref{lem:RealSing} and Theorem \ref{thm:levelwiseequivalence}, the maps
\[
\norm{\delta Y_\bullet} \overset{D}\lra \norm{\Sing_\bullet Y_\bullet} \lra \norm{Y_\bullet}
\]
are weak equivalences.

\begin{lem}\label{groupcompletionproof-discretizationlemma}
If $Y_\bullet$ is a simplicial space such that $\delta Y_\bullet$ is basic, then $Y_\bullet$ is basic.
\end{lem}
\begin{proof}
The proof uses the simplicial structure in an essential way. As in the second proof of Lemma \ref{lem:simplicial-hocartesianness} we may assume that $f_p : X_p \to Y_p$ is a fibration for each $p$. Let $Y_{p,q} := \Sing_q Y_p$, giving a bi-simplicial set $Y_{\bullet, \bullet}$, and define a bi-simplicial space $X_{\bullet,\bullet}$ and a map $f_{\bullet,\bullet} : X_{\bullet,\bullet} \to Y_{\bullet,\bullet}$ as follows. Let $X_{p,q} := \coprod_{\sigma \in Y_{p,q}} \lift (\sigma,f_p)$, where $\lift (\sigma,f_p)$ is the space of all maps $h: \Delta^q \to X_p$ with $f_p \circ h = \sigma$, equipped with the compact-open topology. The simplicial structure in the $p$ direction is given by $ h \mapsto d_i \circ h$ and in the $q$-direction by $h \mapsto h \circ d^j$ (similarly for the degeneracy maps). The evident maps $f_{p,q}:X_{p,q} \to Y_{p,q}$ are the components of a bi-simplicial map. Because $f_p$ is a fibration, the map $f_{p,q}^{-1} (\sigma) \to f_{p,q-1}^{-1}(d_i \sigma)$ is a weak equivalence, for each 
$q$ and $i$. Hence the simplicial map $f_{p,
\bullet}: X_{p,\bullet} \to Y_{p,\bullet}$ is homotopy cartesian.

Analogous to the evaluation map $\norm{Y_{p,\bullet}} \to Y_p$, let $u_p: \norm{X_{p,\bullet}} \to X_p$ be the map which sends $(h,t) \in X_{p,q} \times \Delta^q$ to $h(t)\in X_p$. These are the components of a map of simplicial spaces, and the diagram
\[
\xymatrix{
X_{p,0} \ar[d]^{f_{p,0}} \ar[r] & \norm{X_{p,\bullet}} \ar[r]^{u_p} \ar[d]^-{\norm{f_{p,\bullet}}}& X_p \ar[d]^{f_p}\\
Y_{p,0} \ar[r] & \norm{Y_{p,\bullet}} \ar[r] & Y_p
}
\]
commutes. As $f_{p,\bullet}$ is homotopy cartesian, it follows from Theorem \ref{lem:simplicial-hocartesianness} that the left-hand square is homotopy cartesian. The space $Y_{p,0}$ is $Y_p$ with the discrete topology, and $f_{p,0}^{-1} (y)= f_p^{-1}(y)$. Therefore, the outer rectangle is homotopy cartesian. Moreover, $Y_{p,0} \to \norm{Y_{p,\bullet}}$ is $0$-connected, so it follows that the right-hand square is homotopy cartesian as well. The bottom right-hand map is a weak equivalence by Lemma \ref{lem:RealSing}, so the map $u_p$ is also a weak equivalence.

So far, we set the stage for the following diagonal argument. Consider the commutative square
\[
\xymatrix{
X_{0,0} \ar[d]^{f_{0,0}} \ar[r] & \norm{\delta X_{\bullet}} \ar[d]^{\delta f_\bullet} \ar[r]^{\simeq}  & \norm{X_{\bullet,\bullet}} \ar[d]^{\norm{f_{\bullet,\bullet}}} \ar[r]_{\norm{u_\bullet}}^{\simeq} & \norm{X_\bullet} \ar[d]^{\norm{f_\bullet}}\\
Y_{0,0} \ar[r] & \norm{\delta Y_{\bullet}}  \ar[r]^{\simeq} & \norm{Y_{\bullet,\bullet}} \ar[r]^{\simeq} & \norm{Y_\bullet},
}
\]
where the weak equivalences in the middle come from Theorem \ref{thm:eilenberg-zilber}. 

Since $f_p$ is a fibration, and the original map $f_\bullet$ was $\cA$-cartesian, it follows that $f_{\bullet,\bullet}$ is $\cA$-cartesian (in the obvious sense: we require that the diagrams in Definition \ref{defn:hocartesian.morpism} to be $\cA$-cartesian in both simplicial directions), and hence that $\delta f_\bullet$ is $\cA$-cartesian. By the hypothesis of the lemma, $\delta Y_\bullet$ is basic and so the left square is $\cA$-cartesian. 
Since the other horizontal maps are weak equivalences, it follows that the outer rectangle is $\cA$-cartesian, which concludes the proof.
\end{proof}

The next step is to show that the property of being basic descends along homotopy cartesian maps. 

\begin{lem}\label{groupcompletionproof-lemma3}
Let $h_\bullet:Z_\bullet \to Y_\bullet$ be a homotopy cartesian map of simplicial spaces and assume that $h_0$ is $0$-connected. If $Z_\bullet$ is basic then $Y_\bullet$ is basic.
\end{lem}

\begin{proof}
The proof only uses the semi-simplicial structure. 
Let $f_\bullet:X_\bullet \to Y_\bullet$ be a $\cA$-cartesian map of simplicial spaces. As in the second proof of Lemma \ref{lem:simplicial-hocartesianness} we may assume that each $f_p$ is a fibration. We form the levelwise pullback
\begin{equation}\label{eqn1.groupcompletionproof-lemma3}
\xymatrix{
 W_p \ar[r]^{k_p} \ar[d]^{g_p} & X_p \ar[d]^{f_p} \\
 Z_p \ar[r]^{h_p} & Y_p,
 }
\end{equation}
and this diagram is homotopy cartesian, because $f_p$ is a fibration. The map $g_\bullet$ is $\cA$-cartesian. To see this, let $z \in Z_p$ be a point and consider the commutative diagram
\[
 \xymatrix{
\hofib_z (g_p) \ar[d]^{\simeq} \ar[r]& \hofib_{d_i z}(g_{p-1}) \ar[d]^{\simeq}\\
 \hofib_{h_p(z)}(f_p) \ar[r] & \hofib_{d_i h_p(z)} (f_{p-1})
 }
\]
and use that $\cA$-equivalences satisfy the $2$-out-of-$3$ property. A similar argument (using also Remark \ref{remark.hocartesianness-symmetric}) shows that $k_\bullet$ is homotopy cartesian. The square
\[
 \xymatrix{
 \norm{W_\bullet} \ar[r]^{\norm{k_\bullet}} \ar[d]^{\norm{g_\bullet}} & \norm{X_\bullet} \ar[d]^{\norm{f_\bullet}} \\
 \norm{Z_\bullet} \ar[r]^{\norm{h_\bullet}} & \norm{Y_\bullet}
 }
\]
is homotopy cartesian. This follows by applying Theorem \ref{lem:simplicial-hocartesianness} to both $h_\bullet$ and $k_\bullet$, using that \eqref{eqn1.groupcompletionproof-lemma3} is homotopy cartesian for $p=0$ and using Remark \ref{remark.hocartesianness-symmetric}. Since \eqref{eqn1.groupcompletionproof-lemma3} for $p=0$ is homotopy cartesian, comparing homotopy fibres gives a commutative square
\[
 \xymatrix{
\ar[d] \hofib_{z} (g_0) \ar[r]^{\simeq} &\hofib_{h_0 (z)} (f_0) \ar[d] \\
 \hofib_{\iota(z)} (\norm{g_\bullet}) \ar[r]^{\simeq} & \hofib_{\iota(z)} (\norm{f_\bullet})
 }
\]
in which the horizontal maps are weak equivalences. Since $g_\bullet : W_\bullet \to Z_\bullet$ is $\cA$-cartesian, and by assumption $Z_\bullet$ is basic, it follows that the left vertical map is an $\cA$-equivalence. Therefore, the right vertical map is also an $\cA$-equivalence. This holds for any $z \in Z_0$, but the map $h_0$ is $0$-connected, which finishes the proof.
\end{proof}

The next lemma provides an appropriate resolution of a simplicial \emph{set} by a contractible simplicial space.

\begin{lem}\label{lem:groupcompletion-resolution}
Let $Y_\bullet$ be a $0$-connected simplicial set. Then there is a simplicial space $QY_\bullet$ with $\norm{QY_\bullet}\simeq *$ and a homotopy cartesian morphism $f_\bullet:QY_\bullet \to Y_\bullet$, such that $f_0$ is $0$-connected.
\end{lem}

The same statement is true for semi-simplicial sets, with the same proof. 

\begin{proof}
Fix a vertex $y\in Y_0$. For each simplex $\sigma \in Y_p$, we let $\chi_\sigma: \Delta^p \to \norm{Y_\bullet}$ denote its characteristic map. Furthermore, we view $\Delta^p \subset \Delta^{p+1}$ as the last face, i.e. the face opposite to $e_{p+1}$. We let 
\[
QY_p:= \coprod_{\sigma \in Y_p} \{(\sigma,h)\vert h: \Delta^{p+1} \to \norm{Y_\bullet} ; \; h|_{\Delta^p}= \chi_\sigma; \; h(e_{p+1})= y \},
\] 
topologised as a subspace of $Y_p \times \norm{Y_\bullet}^{\Delta^p}$.
Define $d_i: QY_p \to QY_{p-1}$ by $d_i (\sigma,h):= (d_i \sigma, h \circ d^i)$ (and the degeneracy maps in an analogous way) and $f_p: QY_p \to Y_p$ by $f_p (\sigma,h):= \sigma$. Then $f_\bullet:QY_\bullet \to Y_\bullet$ is a map of simplicial spaces. 

This should be viewed as an analogue of the path fibration, and we now verify that indeed it has the characteristic properties of that construction. The maps $d_i:f_p^{-1}(\sigma)\to f_{p-1}^{-1}(\sigma)$ are homotopy equivalences, so that $f_\bullet$ is homotopy cartesian. It remains to be shown that $\norm{QY_\bullet}\simeq *$.

First observe that the fibre $f_0^{-1}(y)$ is the based loop space $\Omega_y \norm{Y_\bullet}$. 
Let $P_y \norm{Y_\bullet}$ denote the path space: the space of all paths in $\norm{Y_\bullet}$ with endpoint $y$. The map
\begin{align*}
g: \norm{QY_\bullet} &\lra P_y \norm{Y_\bullet}\\
 (\sigma,h,t) &\longmapsto \Bigl(s \mapsto h ((1-s)t,s) \Bigr)
\end{align*}
makes the diagram
\[
\xymatrix{
\Omega_y \norm{Y_\bullet} \ar[r] \ar[d] & \norm{QY_\bullet} \ar[r]^g \ar[d]^{\norm{f_\bullet}} & P_y \norm{Y_\bullet} \ar[d]^{\mathrm{ev}_0} \\
\{y\} \ar[r] & \norm{Y_\bullet} \ar@{=}[r] & \norm{Y_\bullet}
}
\]
commute, by inspection. Since $f_\bullet$ is homotopy cartesian, the left-hand square is homotopy cartesian. The outer rectangle is also hompotopy cartesian, as it is cartesian and $\mathrm{ev}_0$ is a fibration. Thus the map between vertical homotopy fibres over $y$ of the right-hand square is an equivalence: this holds for all $y$, so the right-hand square is homotopy cartesian, and hence $g$ is a weak equivalence. Thus $\norm{QY_\bullet} \simeq *$ as desired.
\end{proof}

The deduction of Proposition \ref{prop:groupcompletion-simplicialanalog} is fairly easy.

\begin{proof}[Proof of Proposition \ref{prop:groupcompletion-simplicialanalog}]
We have to show that every simplicial space $Y_\bullet$ is basic. It is no loss of generality to assume that $\norm{Y_\bullet}$ is $0$-connected. Using the construction from Lemma \ref{lem:groupcompletion-resolution}, we consider the simplicial set $\delta (Q(\delta Y))_\bullet$. This is contractible (by Lemma \ref{lem:groupcompletion-resolution} and Theorem \ref{thm:eilenberg-zilber}) so by Lemma \ref{groupcompletionproof-lemma1} $\delta (Q(\delta Y))_\bullet$ is basic. By Lemma \ref{groupcompletionproof-discretizationlemma}, it follows that $Q(\delta Y)_\bullet$ is basic. As the map $f_\bullet : Q(\delta Y)_\bullet \to \delta Y_\bullet$ provided by Lemma \ref{lem:groupcompletion-resolution} is homotopy cartesian and $f_0$ is 0-connected, it follows from Lemma \ref{groupcompletionproof-lemma3} that $\delta Y_\bullet$ is basic. Finally, using Lemma \ref{groupcompletionproof-discretizationlemma} again, it follows that $Y_\bullet$ is basic.
\end{proof}

\begin{proof}[Proof of Theorem \ref{thm:Acart}]
By Proposition \ref{prop:groupcompletion-simplicialanalog}, every simplicial space is basic. We will make use of the functor $E : \se \Top \to \si \Top$ which freely adds degeneracies. Let $f_\bullet: X_\bullet \to Y_\bullet$ be an $\cA$-cartesian map of semi-simplicial spaces, giving a map $Ef_\bullet : EX_\bullet \to EY_\bullet$ of simplicial spaces. It follows from the description of the simplices and face maps of $EY_\bullet$ that $Ef_\bullet$ is also $\cA$-cartesian. Consider the commutative diagram
\[
\xymatrix{
X_0 \ar[d]^{f_0} \ar[r] & \norm{X_\bullet} \ar[d]^{\norm{f_\bullet}} \ar[r] & \norm{EX_\bullet} \ar[d]^{\norm{Ef_\bullet}} \\
Y_0 \ar[r] & \norm{Y_\bullet}  \ar[r] & \norm{EY_\bullet}.
}
\]
The simplicial space $EY_\bullet$ is basic by Proposition \ref{prop:groupcompletion-simplicialanalog}, so as $EX_0=X_0$ and $EY_0=Y_0$ we have that the outer rectangle is $\cA$-cartesian. As the two rightmost horizontal maps are weak equivalences, by Lemma \ref{lem:AdjUnitsTop}, it follows that the left-hand square is $\cA$-cartesian, as claimed. 
\end{proof}

\subsection{Group-completion}

Let us describe the application of Theorem \ref{thm:Acart} to group-completion. Let $M$ be a (topological) monoid acting on the left on a space $X$ and on the right on a space $Y$. One may form the \emph{two-sided bar construction} $B_\bullet(Y, M, X)$, the semi-simplicial space having $p$-simplices $Y \times M^{p} \times X$, with face maps
\begin{align*}
d_0(y, m_1, \ldots, m_p, x) &= (y \cdot m_1, m_2, \ldots, m_p, x)\\
d_i(y, m_1, \ldots, m_p, x) &= (y, m_1, \ldots, m_{i-1}, m_i \cdot m_{i+1}, m_{i+2}, \ldots, m_p, x) \text{ for } 0 < i <p \\
d_p(y, m_1, \ldots, m_p, x) &= (y, m_1, m_2, \ldots, m_p \cdot x).
\end{align*}
Now let $Y=*$ and suppose that $M$ acts on $X$ by $\cA$-equivalences. Then the projection map $B_\bullet(*, M, X) \to B_\bullet(*, M, *)$ is $\cA$-cartesian, and so by Theorem \ref{thm:Acart} the square
\begin{equation}\label{eq:Acart}
\begin{gathered}
 \xymatrix{
 X \ar[r] \ar[d]^{f_0} & \norm{B_\bullet(*, M, X)} \ar[d]^{\norm{f_\bullet}} \\
 \{*\} \ar[r] & \norm{B_\bullet(*, M, *)} \ar@{=}[r]&  BM
 }
\end{gathered}
\end{equation}
is $\cA$-cartesian.

We apply this as follows. Suppose that the set of path-components of $M$ is countable and let $m_1, m_2, m_3, \ldots \in M$ be a sequence of points  with infinitely-many in each path component. We may form the homotopy colimit
$$M_\infty = \mathrm{hocolim} (M \overset{- \cdot m_1}\to M \overset{- \cdot m_2}\to M \overset{- \cdot m_3}\to \cdots)$$
over right multiplication in the monoid $M$ by the $m_i$; this has a residual left $M$-action. If the monoid $M$ is homotopy commutative, then $H_*(M;\bZ)$ has the structure of a commutative ring, and we can identify
$$H_*(M_\infty;\bZ) \cong \colim (H_*(M;\bZ) \overset{(- \cdot m_1)_*}\to H_*(M;\bZ) \overset{(- \cdot m_2)_*}\to H_*(M;\bZ) \overset{(- \cdot m_3)_*}\to \cdots)$$
with the localisation $H_*(M;\bZ)[\pi_0(M)^{-1}]$ of the ring $H_*(M)$ at the multiplicative subset $\pi_0(M) \subset H_0(M;\bZ)$. In particular, the map $m \cdot - : M_\infty \to M_\infty$ given by left multiplication by $m$ induces an isomorphism on homology. We may thus apply the above observation to the left action of $M$ on $M_\infty$. Now $B_\bullet(*, M, M)$ has an extra degeneracy (as in Lemma \ref{lem:ExtraDeg}), so $\norm{B_\bullet(*, M, M)} \simeq *$ and hence
$$\norm{B_\bullet(*, M, M_\infty)} \simeq \mathrm{hocolim} (\norm{B_\bullet(*, M, M)} \overset{- \cdot m_1}\to \norm{B_\bullet(*, M, M)} \to \cdots ) \simeq *.$$
The homology-cartesian square \eqref{eq:Acart} therefore gives a map
\begin{equation}\label{eq:GCT}
M_\infty \lra \hofib_*(\norm{B_\bullet(*, M, X)} \to \norm{B_\bullet(*, M, *)}) \simeq \Omega BM
\end{equation}
which is an integral homology equivalence; in particular
$$H_*(M;\bZ)[\pi_0(M)^{-1}] \cong H_*(\Omega BM ; \bZ).$$

\begin{rem}
In fact, the argument of \cite{RWGC} shows that in the situation above the monoid $M$ acts on $M_\infty$ by abelian homology equivalences, and so the map \eqref{eq:GCT} is an abelian homology equivalence, but the fundamental group of the target is abelian, so it follows that \eqref{eq:GCT} is in fact an acylic map.
\end{rem}

There is also a group-completion theorem for categories, rather than monoids: it can also be deduced immediately from Theorem \ref{thm:Acart}; we refer the reader to \cite[Section 7]{GMTW} for a formulation.

\section{Products of simplicial spaces}\label{sec:products}

Let $X_{\bullet,\bullet}$ be a bi-simplicial space and let $\delta (X_{\bullet,\bullet})$ be the diagonal simplicial space. To define the \emph{diagonal map} $D: \norm{\delta(X_{\bullet,\bullet})} \to \norm{X_{\bullet,\bullet}}$, take the diagonal map $d:\Delta^p \to \Delta^p \times \Delta^p$ and 
\[
(\id_{X_{p,p}} \times d):  X_{p,p} \times \Delta^p \lra X_{p,p} \times \Delta^p \times \Delta^p.
\]
This respects the equivalence relations used for the definition of the fat geometric realisation and so induces a map $D$ as indicated. 

\begin{thm}\label{thm:eilenberg-zilber}
The diagonal map $D$ is a weak equivalence.
\end{thm}

This is false if one considers bi-semi-simplicial spaces instead: if $Y_\bullet$ is an arbitrary semi-simplicial space and $X_{\bullet,\bullet} = \nabla^0_\bullet \otimes Y_\bullet$, then $\norm{X_{\bullet,\bullet}} = \norm{Y_\bullet} $ and $\norm{\delta(X)_\bullet} = Y_0$. Let us note an application of Theorem \ref{thm:eilenberg-zilber}.

\begin{thm}\label{thm:products-of-simplicial-spaces}
Let $X_\bullet$ and $Y_\bullet$ be simplicial spaces. Then the map
\[
\norm{(X\times Y)_\bullet} \lra \norm{X_\bullet} \times \norm{Y_\bullet},
\]
induced by the two projection maps $(X\times Y)_\bullet \to X_\bullet$ and $(X \times Y)_\bullet \to Y_\bullet$, is a weak homotopy equivalence.
\end{thm}

\begin{proof}
The diagram
\[
\xymatrix{
\norm{(X\times Y)_\bullet} \ar@{=}[d] \ar[r] & \norm{X_\bullet} \times \norm{Y_\bullet} \\
\norm{\delta (X \otimes Y)_\bullet} \ar[r]^{\simeq} & \norm{X_\bullet \otimes Y_\bullet} \ar[u]^{\cong}
}
\]
commutes, and the indicated homeomorphism and weak equivalence are true by Theorem \ref{thm:eilenberg-zilber} and \eqref{geometric-realization-bisemisimplicial2}.
\end{proof}

Theorem \ref{thm:products-of-simplicial-spaces} is important when one applies Segal's theory of $\Gamma$-spaces to deloop spaces which arise as geometric realisations of simplicial spaces. This will be done in \cite{ERW17} and has been done at various places in the literature.

One could derive Theorem \ref{thm:eilenberg-zilber} from the classical result \cite[Theorem I.3.7]{GelMan} that for a bi-simplicial set, one has a homeomorphism $|\delta (X_{\bullet,\bullet})| \cong |X_{\bullet,\bullet}|$ and from \cite[Proposition A.1]{Segal}. However, it seems to be easier to give an argument from scratch. The main bulk of work for the proof of Theorem \ref{thm:eilenberg-zilber} is the proof for bi-simplicial \emph{sets}, and the proof of that case resembles in some sense the proof of classical Eilenberg--Zilber theorem in singular homology, using the method of acyclic models. The first step is to prove that the ``models'' are contractible.

\begin{lem}\label{eilenberg-zilber-proof1}
Let $\Delta^{n,m}_{\bullet,\bullet} := \Delta_\bullet^n \otimes \Delta_\bullet^m$ be the ``bi-simplicial $(n,m)$-simplex''. The spaces $\norm{\Delta^{n,m}_{\bullet,\bullet}}$ and $\norm{\delta (\Delta^{n,m})_{\bullet}}$ are contractible. In particular, Theorem \ref{thm:eilenberg-zilber} is true when $X_{\bullet,\bullet}=\Delta^{n,m}_{\bullet,\bullet}$.
\end{lem}

\begin{proof}
By \eqref{geometric-realization-bisemisimplicial2} and Example \ref{ex:realization-of-simplicial-simplex}, we have
\[
\norm{\Delta^{n,m}_{\bullet,\bullet}} \cong \norm{\Delta^n_\bullet} \times \norm{\Delta^m_\bullet} \simeq *.
\]
To prove that $\norm{\delta (\Delta^{n,m})_\bullet}\simeq *$, consider the ordered set $[n]$ as a (unital) category. Then $\Delta_\bullet^n$ is the nerve of $[n]$. Moreover, $\delta (\Delta^{n,m})_\bullet$ is the nerve of the category $[n] \times [m]$. This category has a terminal object, namely $(n,m)$, so a natural transformation from the identity functor to a constant functor. It follows from Lemma \ref{lem:NatTrans} that $\norm{\delta (\Delta [n,m])_\bullet}$ is contractible.
\end{proof}

It is in this step that the degeneracies are used. The analogous claim for bi-semi-simplicial sets is false. The role of $\Delta^{n,m}_{\bullet,\bullet}$ is then taken by $\nabla_{\bullet,\bullet}^{n,m}:=\nabla^n_\bullet \otimes \nabla_\bullet^m$. While $\norm{\nabla_{\bullet,\bullet}^{n,m}}$ is contractible, $\norm{\delta (\nabla^{n,m})_\bullet}$ usually is not. This may be seen by calculating the Euler number of these finite complexes. 

The identity $\id_{[n]}$ defines an element $\iota_n \in \Delta^n_n$ and its characteristic map $\widehat{\iota_n}: \Delta^n \to \norm{\Delta^n_\bullet}^{(n)} \subset \norm{\Delta^n_\bullet}$. The restriction to the topological boundary $\partial \Delta^n$ goes into the $(n-1)$-skeleton $\norm{\Delta^n_\bullet}^{(n-1)}$ and is denoted $\partial \widehat{\iota_n}$. 
In a similar vein, the tautological element $\iota_{n,m}=(\iota_n, \iota_m) \in \Delta^{n,m}_{n,m}$ induces a map $\widehat{\iota_{n,m}}: \Delta^n \times \Delta^m \to \norm{\Delta^{n,m}_{\bullet,\bullet}}^{n+m}$ with boundary $\partial \widehat{\iota_{n,m}} : (\Delta^n \times \partial \Delta^m \cup \partial \Delta^n \times \Delta^m)=: \partial (\Delta^n \times \Delta^m) \to \norm{\Delta^{n,m}_{\bullet,\bullet}}^{n+m-1}$. Moreover, composition with the diagonal map $d:\Delta^n \to \Delta^n \times \Delta^n$ (whose restriction to $\partial \Delta^n$ goes into $\partial (\Delta^n \times \Delta^n)$) defines a map $\widehat{\iota_{n,n}} \circ d: \Delta^n \to \norm{\delta(\Delta^{n,n})_\bullet}^{(n)}$, with boundary map $\partial (\widehat{\iota_{n,n}} \circ d): \partial \Delta^n \to \norm{\delta(\Delta^{n,n})_\bullet}^{(n-1)}$. 

Note that $X_{\bullet,\bullet} \mapsto \norm{\delta (X)_\bullet}$ and $X_{\bullet,\bullet} \mapsto \norm{X_{\bullet,\bullet}}$ are functors from the category of bi-simplicial sets to $\Top$ and the diagonal map $D$ is a natural transformation. Moreover, both $\norm{\delta (X)_\bullet}$ and $\norm{X_{\bullet,\bullet}}$ are naturally filtered spaces, their $0$-skeleta are equal:
\[
 \norm{\delta (X)_\bullet}^{(0)} =\norm{X_{\bullet,\bullet}}^{(0)}= X_{0,0},
\]
and $D$ restricts to the identity between the $0$-skeleta. 

\begin{lem}\label{eilenberg-zilber-proof2}\mbox{}
\begin{enumerate}[(i)]
\item There is a natural map $F: \norm{X_{\bullet,\bullet}} \to \norm{\delta (X)_\bullet} $ which is the identity on the $0$-skeleton.
\item The map $D \circ F: \norm{X_{\bullet,\bullet}} \to \norm{X_{\bullet,\bullet}}$ is naturally homotopic to the identity.
\item The map $F \circ D: \norm{\delta (X)_{\bullet}} \to \norm{\delta (X)_{\bullet}}$ is naturally homotopic to the identity.
\end{enumerate}
In particular, $D$ is a homotopy equivalence, for each bi-simplicial set.
\end{lem}

One can add the statements that the maps $F$ and $D$ are unique up to natural homotopy among those natural maps which are the identity on the $0$-skeleton. These statements will not enter the proof of Theorem \ref{thm:eilenberg-zilber} and so we do not prove them, but the method of proof can easily be adapted.

\begin{proof}
We shall construct the map $F$ and the homotopies inductively on skeleta. More precisely, we shall construct natural maps 
\[
F_n =F_n^X : \norm{X_{\bullet,\bullet}}^{(n)} \lra \norm{\delta(X)_\bullet} 
\]
and natural homotopies
\[
h_n: F \circ D_n \leadsto \id, \; k_n : D \circ F_n \leadsto \id. 
\]
We begin with the construction of $F_n$. The map $F_0$ is the identity, and we assume that $F_0, \ldots, F_{n-1}$ are already constructed. Let $p+q =n$, and we first construct a suitable map $\mu_{p,q}: \norm{\Delta^{p,q}_{\bullet,\bullet}}^{(n)} \to \norm{\delta(\Delta^{p,q})_\bullet}$. The inclusion map $\norm{\Delta^{p,q}_{\bullet,\bullet}}^{(n-1)} \to \norm{\Delta^{p,q}_{\bullet,\bullet}}^{(n)}$ is a cellular inclusion. By Lemma \ref{eilenberg-zilber-proof1}, the space $\norm{\delta (\Delta^{p,q})_\bullet}$ is contractible. Hence there exists a solution $\mu_{p,q}$ to the extension problem
\[
 \xymatrix{
 \norm{\Delta^{p,q}_{\bullet,\bullet}}^{(n-1)}  \ar[r]^-{F_{n-1}^{\Delta^{p,q}}} \ar[d]  & \norm{\delta(\Delta^{p,q})_\bullet} \\
 \norm{\Delta^{p,q}_{\bullet,\bullet}}^{(n)} \ar@{..>}[ur]_{\mu_{p,q}}. &  
 }
\]
Now we construct $F_n^{X}$ for a bi-simplicial set $X$. Observe that
\[
 X_{p,q} = \bis \Set (\Delta_{\bullet,\bullet}^{p,q} ,X_{\bullet,\bullet}),
\]
the set of morphisms of bi-simplicial sets (this is an instance of the Yoneda lemma). For each $s\in X_{p,q}$, we have the characteristic map $\widehat{s}: \Delta^p \times \Delta^q \to \norm{X_{\bullet,\bullet}}^{p+q}$, and if we view $s$ as a map of bi-simplicial sets, we obtain $\norm{s}: \norm{\Delta^{p,q}_{\bullet,\bullet}} \to \norm{X_{\bullet,\bullet}}$. The relation between these two maps is that $\norm{s} \circ \widehat{\iota_{p,q}} = \widehat{s}$.
The following diagram is a pushout diagram
\[
 \xymatrix{
 \coprod\limits_{\substack{p+q=n\\  s \in X_{p,q}}} \partial \Delta^{p,q} \ar[r]^-{\varphi} \ar[d]^{\inc} & \norm{X_{\bullet,\bullet}}^{(n-1)} \ar[d]\\
 \coprod\limits_{\substack{p+q=n\\  s \in X_{p,q}}} \Delta^{p,q}  \ar[r]^-{\phi} & \norm{X_{\bullet,\bullet}}^{(n)},
 }
\]
where the map $\phi$ is 
\[
\phi=\coprod_{\substack{p+q=n\\  s \in X_{p,q}}} \widehat{s} = \coprod\limits_{\substack{p+q=n\\  s \in X_{p,q}}} \norm{s}^{(n)} \circ \widehat{\iota_{p,q}}
\]
and similarly
\[
\varphi = \coprod_{\substack{p+q=n \\  s \in X_{p,q}}} \norm{s}^{(n-1)} \circ \partial \widehat{\iota_{p,q}}.
\]
We claim that the two maps 
\[
F_{n-1}^{X_{\bullet,\bullet}} \circ \varphi, \coprod\limits_{\substack{p+q=n\\  s \in X_{p,q}}} \norm{\delta(s)_\bullet} \circ \mu_{p,q} \circ \widehat{\iota_{p,q}} \circ \inc:  \coprod\limits_{\substack{p+q=n\\  s \in X_{p,q}}} \norm{\Delta^{p,q}_{\bullet,\bullet}}^{(n-1)}  \lra \norm{\delta(X)_\bullet} 
\]
are equal, and so they induce a map from the pushout, i.e. $F_n^X: \norm{X_{\bullet,\bullet}}^{(n)} \to \norm{\delta (X)_\bullet}$, which finishes the inductive construction of $F_n$ (it is obvious that $F_n$ becomes a natural map). To verify this claim, it is enough to check that for each $s \in X_{p,q}$, the diagram
\[
\xymatrix{
\norm{\Delta^{p,q}_{\bullet,\bullet}}^{(n-1)} \ar[r]^{\norm{s}^{(n-1)}} \ar[d] & \norm{X_{\bullet,\bullet}}^{(n-1)} \ar[dr]^{F_{n-1}^X} & \\
\norm{\Delta^{p,q}_{\bullet,\bullet}}^{(n)} \ar[r]^{\mu_{p,q}} & \norm{\delta(\Delta^{p,q})_\bullet} \ar[r]^{\norm{\delta (s)}} & \norm{\delta(X)_\bullet}
}
\]
commutes. But this is clear because $F_{n-1}$ is a natural transformation:
\[
 F_{n-1}^X \circ \norm{s}^{(n-1)} = \norm{\delta(s)_\bullet} \circ F_{n-1}^{\Delta^{p,q}},
\]
and we constructed $\mu_{p,q}$ so that $\mu_{p,q} \circ \inc = F_{n-1}^{\Delta^{p,q}}$. This finishes the construction of $F$. 

Now we turn to the construction of natural homotopies $h_n: I \times \norm{\delta(X)_\bullet}^{(n)} \to \norm{\delta(X)_\bullet}$ from $F \circ D_n$ to the ``identity'' (i.e. inclusion map). We can take $h_0$ to be the constant homotopy. 
Assume that $h_0, \ldots, h_{n-1}$ are already constructed. As before, we first construct a certain map $\lambda_n : I \times \norm{\delta (\Delta^{n,n})_\bullet}^{(n)} \to \norm{\delta (\Delta^{n,n})_\bullet}$. The inclusion map 
\[
I \times  \norm{\delta(\Delta^{n,n})_\bullet}^{(n-1)} \cup \{0,1\} \times \norm{\delta(\Delta^{n,n})_\bullet}^{(n)} \lra I \times \norm{\delta(\Delta^{n,n})_\bullet}^{(n)}
\]
is a cellular inclusion. We define a map 
\[
 I \times  \norm{\delta(\Delta^{n,n})_\bullet}^{(n-1)} \cup \{0,1\} \times \norm{\delta(\Delta^{n,n})_\bullet}^{(n)} \lra \norm{\delta(\Delta^{n,n})_\bullet}
\]
by taking the homotopy $h_{n-1}^{\Delta^{n,n}}$ on the first part, $F \circ D_{n}$ on $\{0\} \times \norm{\delta(\Delta^{n,n})_\bullet}^{(n)}$ and the ``identity'' on $\{1\} \times \norm{\delta(\Delta^{n,n})_\bullet}^{(n)}$. Those fit together by assumption and so define a continuous map. It can be extended to a map
\[
 \lambda_n:  I \times \norm{\delta(\Delta^{n,n})_\bullet}^{(n)} \lra \norm{\delta(\Delta^{n,n})_\bullet},
\]
because the target space is contractible by Lemma \ref{eilenberg-zilber-proof1}. There is a pushout diagram
\[
\xymatrix{
\coprod\limits_{s \in X_{n,n}} I \times \partial \Delta^n \ar[r]^{\varphi} \ar[d]^{\inc} & I \times \norm{\delta(X)_\bullet}^{(n-1)} \ar[d] \\
\coprod\limits_{s \in X_{n,n}} I \times \Delta^n \ar[r]^{\phi} & I \times \norm{\delta(X)_\bullet}^{(n)}
}
\]
whose horizontal maps are given by 
\[
 \varphi = \id_I \times \left(\coprod_{s \in X_{n,n}} \norm{\delta (s)}^{(n-1)} \circ \partial  \widehat{\iota_n}\right)
\]
and
\[
 \phi=  \id_I \times \left(\coprod_{s \in X_{n,n}} \norm{\delta (s)}^{(n)} \circ \widehat{\iota_n}\right). 
\]
Let $\psi: \coprod\limits_{s \in X_{n,n}} I \times \Delta^n \to \norm{\delta(X)_\bullet}$ be the map 
\[
 \coprod_{s \in X_{n,n}} \norm{s} \circ \lambda_n \circ (\id_I \times \widehat{\iota_n}).
\]
Then $\psi \circ \inc = h_{n-1}^X \circ \varphi$ by construction, and so these maps together induce a map $h_n^X$ from the pushout $ I \times \norm{\delta(X)_\bullet}^{(n)}$ to $\norm{\delta(X)_\bullet}$ which extends $h_{n-1}^X$ and is natural. 

The construction of the homotopies $k_n$ is very similar and left to the reader. 
\end{proof}

\begin{proof}[Proof of Theorem \ref{thm:eilenberg-zilber}]
Consider the trisimplicial set $(p,q,r) \mapsto \Sing_r X_{p,q}$. The following diagram commutes:
\[
 \xymatrix{
 \norm{p \mapsto X_{p,p}} \ar[r]^{D} & \norm{(p,q) \mapsto X_{p,q}} \\
 \norm{ p \mapsto \norm{r \mapsto \Sing_r X_{p,p}}} \ar[r]^{D} \ar[u] \ar[d]^{\cong} & \norm{(p,q) \mapsto \norm{r \mapsto \Sing_r X_{p,q}}} \ar[u] \ar[d]^{\cong}\\
 \norm{r \mapsto \norm{p \mapsto \Sing_r X_{p,p}}} \ar[r]^{D} & \norm{r \mapsto \norm{(p,q) \mapsto \Sing_r X_{p,q}}}.
 }
\]
The upper vertical maps are weak equivalences, by Lemma \ref{lem:RealSing} and Theorem \ref{thm:levelwiseequivalence}. The lower vertical maps are the homeomorphisms from \eqref{geometric-realization-bisemisimplicial}. The bottom horizontal map is a weak equivalence by Lemma \ref{eilenberg-zilber-proof2} and by Theorem \ref{thm:levelwiseequivalence}. Hence so is the upper horizontal map, which proves the claim.
\end{proof}

\bibliographystyle{plain}
\bibliography{literature}

\end{document}